\documentclass[12pt, reqno]{amsart}
\usepackage{amssymb}
\usepackage{eucal}
\usepackage{enumerate}
\usepackage{amsmath}
\usepackage{amscd}
\usepackage[dvips]{color}
\usepackage{mathtools}
\usepackage{multicol}
\usepackage[all]{xy}           %xypic macro for latex2.09
\usepackage{graphicx}
\usepackage{txfonts}
\usepackage{color}
\usepackage{colordvi}
\usepackage{xspace}
\usepackage{ulem}
\usepackage{cancel}
\usepackage{tikz}
\usepackage{longtable}
\usepackage{multirow}
\usepackage{ifpdf}
\ifpdf
 \usepackage[colorlinks,final,backref=page,hyperindex]{hyperref}
\else
 \usepackage[colorlinks,final,backref=page,hyperindex,hypertex]{hyperref}
\fi

\usepackage{comment}    %% comment out a paragraph
\usepackage{newtxtext}

\begin{document}

\newcommand{\tabincell}[2]{\begin{tabular}{@{}#1@{}}#2\end{tabular}}

\newcommand{\nc}{\newcommand}
\newcommand{\delete}[1]{}

%%%%%%Use the next black to suppress labels
%\delete{
\nc{\mlabel}[1]{\label{#1}}  % Use this to suppress names
\nc{\mcite}[1]{\cite{#1}}  % Use this to suppress names
\nc{\mref}[1]{\ref{#1}}  % Use this to suppress names
\nc{\meqref}[1]{~\eqref{#1}} % Use this to suppress names
\nc{\mbibitem}[1]{\bibitem{#1}} % Use this to show number
%}

%%%%%%%%%%%Use the next block to show labels
\delete{
\nc{\mlabel}[1]{\label{#1}  % Use the next two lines to show names
{\hfill \hspace{1cm}{\bf{{\ }\hfill(#1)}}}}
\nc{\mcite}[1]{\cite{#1}{{\bf{{\ }(#1)}}}}  % Use this lines to show names
\nc{\mref}[1]{\ref{#1}{{\bf{{\ }(#1)}}}}  % Use this lines to show names
\nc{\meqref}[1]{~\eqref{#1}{{\bf{{\ }(#1)}}}} % Use this lines to show names
\nc{\mbibitem}[1]{\bibitem[\bf #1]{#1}} % Use this to show name
}

%%%%%%%%%%%%%%%%%%%%%%%% Statements
\newtheorem{thm}{Theorem}[section]
\newtheorem{lem}[thm]{Lemma}
\newtheorem{cor}[thm]{Corollary}
\newtheorem{pro}[thm]{Proposition}
\newtheorem{conj}[thm]{Conjecture}
\theoremstyle{definition}
\newtheorem{defi}[thm]{Definition}
\newtheorem{ex}[thm]{Example}
\newtheorem{rmk}[thm]{Remark}
\newtheorem{pdef}[thm]{Proposition-Definition}
\newtheorem{condition}[thm]{Condition}

\renewcommand{\labelenumi}{{\rm(\alph{enumi})}}
\renewcommand{\theenumi}{\alph{enumi}}
\renewcommand{\labelenumii}{{\rm(\roman{enumii})}}
\renewcommand{\theenumii}{\roman{enumii}}

\nc{\tred}[1]{\textcolor{red}{#1}}
\nc{\tblue}[1]{\textcolor{blue}{#1}}
\nc{\tgreen}[1]{\textcolor{green}{#1}}
\nc{\tpurple}[1]{\textcolor{purple}{#1}}
\nc{\btred}[1]{\textcolor{red}{\bf #1}}
\nc{\btblue}[1]{\textcolor{blue}{\bf #1}}
\nc{\btgreen}[1]{\textcolor{green}{\bf #1}}
\nc{\btpurple}[1]{\textcolor{purple}{\bf #1}}

\nc{\li}[1]{\tpurple{#1}} \nc{\lir}[1]{\tpurple{Li:#1}}
\nc{\cm}[1]{\tred{CM:#1}} \nc{\gl}[1]{\tblue{GL:#1}}

%new commands

\newcommand{\End}{\text{End}}

\nc{\calp}{\mathcal{P}}
\nc{\calb}{\mathcal{B}}
\nc{\call}{ L}
\nc{\calo}{\mathcal{O}}
\nc{\frakg}{\mathfrak{g}}
\nc{\frakh}{\mathfrak{h}}
\nc{\ad}{\mathrm{ad}}

\nc{\QQ}{\mathbb{Q}}
\nc{\RR}{\mathbb{R}}
\nc{\ZZ}{\mathbb{Z}}

\nc{\ccred}[1]{\tred{\textcircled{#1}}}

\nc{\daop}{dual a-$\mathcal{O}$-operator\xspace}
\nc{\aop}{a-$\mathcal{O}$-operator\xspace}
\nc{\daops}{dual a-$\mathcal{O}$-operators\xspace}
\nc{\aops}{a-$\mathcal{O}$-operators\xspace}
%\nc{\sctpla}{special apre-Leibniz algebra\xspace}
%\nc{\sctplas}{special apre-Leibniz algebras\xspace}
\nc{\sctplab}{special apre-Leibniz bialgebra\xspace}
\nc{\sctplabs}{special apre-Leibniz bialgebras\xspace}
\nc{\apl}{apre-Leibniz algebra\xspace}
\nc{\apls}{apre-Leibniz algebras\xspace}
\nc{\da}{DPL algebra\xspace}
\nc{\das}{DPL algebra\xspace}
\nc{\sa}{SDPL algebra\xspace}
\nc{\sas}{SDPL algebras\xspace}
\nc{\sctplasubs}{special
type-$a$ pre-Leibniz subalgebras\xspace}

\nc{\quadco}{type-$M$\xspace}
\nc{\rep}{
$\big(( L_{\succ_{A}}, R_{\prec_{A}})M,A\big)$
\xspace}
\nc{\repequ}{ $\big(( L_{\succ_{A}}, R_{\prec_{A}})M,A\big)
=(a_{1} L_{\succ_{A}}+a_{2} R_{\prec_{A}},
b_{1} L_{\succ_{A}}+b_{2} R_{\prec_{A}},A)$
\xspace}
\nc{\dualrep}{
$\big(( L^{*}_{\succ_{A}}, R^{*}_{\prec_{A}})M,
A^*\big)$
\xspace}
\nc{\dualrepequ}{ $\big(( L^{*}_{\succ_{A}}, R^{*}_{\prec_{A}})M,
A^*\big)
=(a_{1} L^{*}_{\succ_{A}}+a_{2} R^{*}_{\prec_{A}}, b_{1} L^{*}_{\succ_{A}}+b_{2} R^{*}_{\prec_{A}},A^{*})$
\xspace}
\nc{\sctplas}{special \delete{
\begin{math}
\begin{pmatrix}
1 & -1  \\
-1 & 0
\end{pmatrix}
\end{math}}type-$a$ pre-Leibniz algebras\xspace}
\nc{\Sctplas}{Special \delete{
        \begin{math}
            \begin{pmatrix}
                1 & -1  \\
                -1 & 0
            \end{pmatrix}
    \end{math}}type-$a$ pre-Leibniz algebras\xspace}
\nc{\sctpla}{special \delete{
    \begin{math}
        \begin{pmatrix}
            1 & -1  \\
           -1 & 0
        \end{pmatrix}
    \end{math}}type-$a$ pre-Leibniz algebra\xspace}
\nc{\sctplc}{special \delete{
    \begin{math}
        \begin{pmatrix}
            1 & -1  \\
           -1 & 0
        \end{pmatrix}
    \end{math}}type-$a$ pre-Leibniz coalgebra\xspace}
\nc{\sctplbs}{special \delete{
    \begin{math}
        \begin{pmatrix}
            1 & -1  \\
           -1 & 0
        \end{pmatrix}
    \end{math}}type-$a$ pre-Leibniz bialgebras\xspace}
\nc{\sctplb}{special \delete{
    \begin{math}
        \begin{pmatrix}
            1 & -1  \\
           -1 & 0
        \end{pmatrix}
    \end{math}}type-$a$ pre-Leibniz bialgebra\xspace}
\nc{\Sctplbs}{Special \delete{
    \begin{math}
        \begin{pmatrix}
            1 & -1  \\
           -1 & 0
        \end{pmatrix}
    \end{math}}type-$a$ pre-Leibniz bialgebras\xspace}
\nc{\ctplas}{\delete{
    \begin{math}
        \begin{pmatrix}
            1 & -1  \\
           -1 & 0
        \end{pmatrix}
    \end{math}}type-$a$ pre-Leibniz algebras\xspace}
\nc{\ctpla}{\delete{
    \begin{math}
        \begin{pmatrix}
            1 & -1  \\
           -1 & 0
        \end{pmatrix}
    \end{math}}type-$a$ pre-Leibniz algebra\xspace}
%\nc{\move}[1]{\footnote{#1}}

%%%%%%%%%%%%%%%%%%%%%%%%%%%%%%%%%%%%%%%%%%%%%%%%%%%%%%%%%%%%%%%%%%
\title[Generalized splitting of algebras  and application to Leibniz algebras]{Generalized splitting of algebras with application to a bialgebra structure of Leibniz algebras induced from averaging Lie bialgebras}

    \author{Chengming Bai}
    \address{Chern Institute of Mathematics \& LPMC, Nankai University, Tianjin 300071, China}
    \email{baicm@nankai.edu.cn}

    \author{Li Guo}
    \address{Department of Mathematics and Computer Science, Rutgers University, Newark, NJ 07102, USA}
    \email{liguo@rutgers.edu}

    \author{Guilai Liu}
    \address{Chern Institute of Mathematics \& LPMC, Nankai University, Tianjin 300071, China}
    \email{liugl@mail.nankai.edu.cn}

    \author{Quan Zhao}
    \address{Chern Institute of Mathematics \& LPMC, Nankai University, Tianjin 300071, China}
    \email{zhaoquan@mail.nankai.edu.cn}

\date{\today}

\begin{abstract}
The classical notion of splitting a binary quadratic operad $\calp$ gives the notion of pre-$\calp$-algebras characterized by $\calo$-operators, with pre-Lie algebras as a well-known example.
Pre-$\calp$-algebras give a refinement of the structure of $\calp$-algebras and is critical in the Manin triple approach to bialgebras for $\calp$-algebras.
Motivated by the new types of splitting appeared in recent studies,
this paper aims to extend the classical notion of splitting, by relaxing the requirement that the adjoint actions of the pre-$\calp$-algebra form a representation of the $\calp$-algebra, to allow also linear combinations of the adjoint actions to form a representation.
This yields a whole family of type-$M$ pre-structures,  parameterized by the coefficient matrix $M$ of the linear combinations. Using the duals of the adjoint actions gives another family of splittings.
Similar generalizations are given to the $\calo$-operator characterization of the splitting, and to certain conditions on bilinear forms.
Furthermore, this general framework is applied to determine the bialgebra structure induced from averaging Lie bialgebras, lifting the well-known fact that an averaging Lie algebra induces a Leibniz algebra to the level of bialgebras.
This is achieved by interpreting the desired bialgebra structure for the Leibniz algebra as the one for a special type-$M$ pre-Leibniz algebra for a suitably chosen matrix $M$ in the above family.
\end{abstract}

\subjclass[2020]{
    17A36,  %Automorphisms, derivations, other operators (nonassociative rings and algebras)
    17A40,  %Ternary compositions
    %18M70,  %Algebraic operads, cooperads, and Koszul duality
    %%%%17B:Lie algebras and Lie superalgebras
    17B10,
    17B38, %Yang-Baxter equations and Rota-Baxter operators
   % 17B40, %Automorphisms,derivations,other operators for Lie algebras and super algebras
    17B60, %Lie (super)algebras associated with other structures (associative, Jordan, ect.)
    %17B63,  %Poisson algebras
    17D25,  %Lie-admissible algebras
    %37J39,   %Relations of finite-dimensional Hamiltonian and Lagrangian systems with topology, geometry and differential geometry (symplectic geometry, Poisson geometry, etc.
    %53D17  %Poisson manifolds; Poisson groupoids and algebroids
    18M70.  %Algebraic operads, cooperads, and Koszul duality
}

\keywords{Splitting of algebra, $\mathcal{O}$-operator,
averaging operator, Manin triple, bialgebra, Leibniz algebra}

\maketitle

\vspace{-1.5cm}

\tableofcontents

\vspace{-1.5cm}

\allowdisplaybreaks

\section{Introduction}
This paper introduces a general framework for splitting of algebras and applies it to provide an induced structure from an averaging Lie bialgebra.

\subsection{Classical splitting of algebras, $\calo$-operators and bilinear forms}\label{sec:1.1}\

Let $\calp$ be a binary quadratic operad \cite{LV}. A splitting of $\calp$ is a refinement of $\calp$. The defining operations of the refinement give a splitting of the defining operations of $\calp$ into sums of multiple parts, and the defining relations of the refinement can also be regarded as a splitting of the defining relations of $\calp$ into multiple parts.

Loday \cite{Lod4} described his dendriform algebra as splitting
the associative product into a sum of two products. Over the
years, many types of splittings of operations have
emerged, eventually generalized to the notion of pre-$\calp$-algebras, as well as post-$\calp$-algebras \cite{BBGN}.
This paper will focus on the case when the operad $\calp$ has one
binary operation which is split into two parts. Then the
``classical" splitting of a $\calp$-algebra, called a pre-$\calp$-algebra has the following
precise meaning in addition to its operadic characterization.

\begin{defi} \label{d:split}
\cite{BBGN} For a binary quadratic operad $\calp$ with one binary operation, a {\bf pre-$\calp$-algebra} is defined to be the vector space $A$ together with multiplications $\succ_{A},\prec_{A}:A\otimes A\rightarrow A$ satisfying the following conditions.
\begin{enumerate}
\item\label{321} The triple $(A,\succ_A,\prec_A)$ is {\bf $\calp$-admissible} in the sense that $(A,\circ_A)$ is a $\calp$-algebra, where
$$x\circ_A y:= x\succ_A y+x\prec_A y, \quad \forall x, y\in A.$$
Then $(A,\succ_A,\prec_A)$ is called {\bf compatible on $(A,\circ_A)$},
and $(A,\circ_A)$ is called the {\bf sub-adjacent
$\calp$-algebra} of $(A,\succ_A,\prec_A)$.
\item\label{322} With the linear maps $  L_{\succ_{A}}, R_{\prec_{A}}:A\rightarrow\mathrm{End}_{\mathbb K}(A)$ defined by
\begin{eqnarray*}
    L_{\succ_{A}}(x)y:=x\succ_{A}y,\;
    R_{\prec_{A}}(x)y:=y\prec_{A}x,\;\forall x,y\in A,
\end{eqnarray*}
the triple $( L_{\succ_{A}}, R_{\prec_{A}},A)$ is a representation of $(A,\circ_{A})$.
\end{enumerate}
\end{defi}

Such splitting of an operad reveals finer
properties of $\calp$ which are not expressed in terms of  $\calp$
itself. For example, the commonly used shuffle algebra gains its categorical significance as the free object in the category of pre-commutative algebras (that is, Zinbiel algebras)\,\cite{Lo07}. Likewise, the construction of the free dendriform
algebra by planar binary rooted trees gives the noncommutative
counterpart of the shuffle algebra, usually known as the
Loday-Ronco Hopf algebra \cite{LR} which is also isomorphic to the
noncommutative Connes-Kreimer Hopf algebra~\cite{AS}.
Among its numerous applications in mathematics and
physics \cite{Bai2021.2,Bur,Man}, the pre-Lie algebra as a splitting of the Lie algebra provides a
suitable algebra structure on rooted trees that has appeared in many applications, including vector fields, numerical analysis and
quantum field theory~\cite{Bro,Cay,CK,CL,DL}.

On the other hand, splittings of operads can be achieved by the applications of Rota-Baxter operators and, more generally, the $\calo$-operators (also called relative Rota-Baxter operators). The notion of $\mathcal{O}$-operators was first introduced on Lie algebras as a generalization of the classical Yang-Baxter equation \cite{Ku} and then defined for other algebra structures. For example, an $\calo$-operator of an associative algebra (resp. Lie algebra) gives rise to a dendriform algebra (resp. pre-Lie algebra) \cite{BBGN,PBGN}.
Back to our case of a binary quadratic operad $\calp$ with one operation, we have

\begin{defi}\label{d:oop}
An {\bf
$\mathcal{O}$-operator} of a $\calp$-algebra $(A,\circ_{A})$
associated to a representation $(l_{\circ_{A}},r_{\circ_{A}},V)$
is a linear map $T:V\rightarrow A$ satisfying
\begin{eqnarray}\label{eq:O-operator}
    (Tu)\circ_{A}(Tv)=T\big( l_{\circ_{A}}(Tu)v+r_{\circ_{A}}(Tv)u \big),\;\forall u,v\in V.
\end{eqnarray}
\end{defi}

In fact, having a compatible pre-$\calp$-algebra structure $(A,\succ_A$, $\prec_A)$ on a $\calp$-algebra $(A,\circ_A)$ is characterized as having an invertible $\mathcal O$-operator of $(A,\circ_A)$.
This characterization has a subtle interpretation when the adjoint representation of the $\calp$-algebra has a natural dual
representation (such a $\calp$-algebra is called {\bf proper} in
\cite{Ku2}). That is, for such a proper $\calp$-algebra $(A,\circ_A)$, a compatible pre-$\calp$-algebra structure $(A,\succ_A$,
$\prec_A)$ of $(A,\circ_A)$ is obtained from an invertible $\calo$-operator associated to the dual representation of the adjoint representation of $(A,\circ_A)$, which can be interpreted as a nondegenerate
(usually symmetric or antisymmetric) bilinear form satisfying
certain conditions.
For instance, a Lie algebra equipped with a nondegenerate (antisymmetric) 2-cocycle induces a compatible pre-Lie algebra structure on the Lie algebra \cite{Chu},
and an associative algebra equipped with  a Connes cocycle
induces a compatible dendriform
algebra structure on the associative algebra \cite{Bai2010}.

As we can see below, recent studies have shown the need to consider nondegenerate bilinear forms with other conditions that do not correspond to the classical $\calo$-operators. Thus we have to expand the scope of $\calo$-operators, by allowing more flexible representations.

\subsection{Generalizations of splitting of algebras, $\calo$-operators and  bilinear forms}\label{sec:1.1-}\

From the rapid development in the study of several aspects of
nonassociative algebras, especially in the study of
$\calp$-algebras with specific nondegenerate bilinear forms, several new splittings of $\calp$-algebras for
different kinds of binary quadratic operads $\calp$ have appeared.
For example, the anti-pre-Lie algebras \cite{LB2022}
naturally arose in the study of Witt type Lie algebras, which are obtained from Lie algebras with nondegenerate commutative
$2$-cocycles. Also, in order to lift the fact that an averaging
commutative associative algebra gives rise to a perm algebra to
the bialgebra level, the notion of (special) apre-perm algebras
was introduced in \cite{BGLZ2025}, which are obtained from perm algebras with nondegenerate symmetric left-invariant bilinear forms. In these
cases, the new algebra structures still have the
property \eqref{321} in Definition \ref{d:split}, but the needed representations
in property \eqref{322} have to take the forms of linear combinations
of $ L_{\succ_{A}}$ and $R_{\prec_{A}}$. Also
anti-pre-Lie algebras and (special) apre-perm algebras are
characterized as some analogues of $\mathcal{O}$-operators of Lie algebras and perm algebras respectively.

Motivated by these new phenomena
that require new types of splitting, $\calo$-operators and conditions for nondegenerate bilinear forms,
and with further applications in mind, this study establishes a general notion to split $\calp$-algebras for arbitrary binary quadratic operad $\calp$.
Furthermore, in order to better explore their relationships with $\calp$-algebras, we also would  like to generalize the characterizations of the classical splitting by $\mathcal{O}$-operators to the new splitting.

For a matrix
$M=\begin{pmatrix}
    a_{1} & b_{1}\\
    a_{2} & b_{2}
\end{pmatrix}\in M_{2}(\mathbb{K})$, and maps $f, g$, we use the notation
$$(f,g)M:=(a_1f+a_2g, b_1f+b_2g).$$
To generalize the notion of $\calo$-operators, in their original defining equation~\eqref{eq:O-operator}, we replace
$l_{\circ_A},r_{\circ_A}$ by two linear maps
$\alpha,\beta\rightarrow {\rm End}_{\mathbb K}(A)$ whose linear combination (or its certain dual) gives a representation of $(A,\circ_A)$.

\begin{defi}\label{defi:dual type O-ope-}
    Let  $(A,\circ_{A})$ be a $\mathcal{P}$-algebra. Suppose that $\alpha ,\beta :A\rightarrow\mathrm{End}_{\mathbb K}(V)$ are linear maps and
    $T:V\rightarrow A$ is a linear map that satisfies the following equation:
    \begin{eqnarray*}
        (Tu)\circ_{A}(Tv)=T\big(\alpha(Tu)v+\beta(Tv)u\big),\;\forall u,v\in V.
    \end{eqnarray*}
Let $M\in M_2(\mathbb{K})$.
    \begin{enumerate}
\item If $\big((\alpha ,\beta )M,V \big)= (a_{1}\alpha+a_{2}\beta ,b_{1}\alpha+b_{2}\beta,V)$ is a representation of $(A,\circ_{A})$, then we say that $T$ is a
{\bf \quadco $\mathcal{O}$-operator of $(A,\circ_{A})$ associated to $(\alpha ,\beta ,V )$}.
\item If $\big((\alpha^*,\beta^*)M,V^*\big)= (a_{1}\alpha^*+a_{2}\beta^* ,b_{1}\alpha^*+b_{2}\beta^*,V^*)$ is a representation of $(A,\circ_{A})$,  then we say that $T$ is a {\bf dual
        \quadco $\mathcal{O}$-operator of $(A,\circ_{A})$ associated to $(\alpha ,\beta ,V )$}.
    \end{enumerate}
\end{defi}

Corresponding to (dual) type-$M$ $\mathcal{O}$-operators,
we propose two kinds of new splittings of $\mathcal{P}$-algebras, which generalize pre-$\mathcal{P}$-algebras as the classical splitting of $\mathcal{P}$-algebras.

\begin{defi}\label{defi:1.2}
    Let $(A,\succ_{A},\prec_{A})$ be a $\mathcal{P}$-admissible algebra and $(A,\circ_{A}=\succ_{A}+\prec_{A})$ be the sub-adjacent $\mathcal{P}$-algebra.
Let $M\in M_2(\mathbb{K})$.
    \begin{enumerate}
\item \label{it:a1} If  $\big(( L_{\succ_{A}}, R_{\prec_{A}})M,A\big)
=(a_{1} L_{\succ_{A}}+a_{2} R_{\prec_{A}},$
$b_{1} L_{\succ_{A}}
+b_{2} R_{\prec_{A}},A)$ is a representation of $(A,\circ_{A})$, then we say $(A,\succ_{A},\prec_{A})$ is a {\bf \quadco pre-$\mathcal{P}$-algebra}.
\item\label{it:a2} If  $\big(( L^{*}_{\succ_{A}}, R^{*}_{\prec_{A}})M,A^*\big)
=(a_{1} L^{*}_{\succ_{A}}+a_{2} R^{*}_{\prec_{A}},$ $b_{1} L^{*}_{\succ_{A}}
+b_{2} R^{*}_{\prec_{A}},A^{*})$ is a representation of $(A,\circ_{A})$, then we say $(A,\succ_{A},\prec_{A})$ is a {\bf dual \quadco pre-$\mathcal{P}$-algebra}.
    \end{enumerate}
\end{defi}

Note that the classical pre-$\mathcal P$-algebra in Definition~\ref{d:split} is just the special case when $M$ is the identity matrix
$I=\begin{pmatrix}
    1 & 0\\
    0 & 1
\end{pmatrix}$, that is, a  pre-$\mathcal P$-algebra is exactly a
type-$I$ pre-$\mathcal P$-algebra.
%\li{Maybe also give the ones in \cite{BGLZ2025,LB2022} as special cases of $M$?}
 Moreover, by the results in \cite{BGLZ2025,LB2022}, an anti-pre-Lie algebra is a type-$(-I)$ pre-Lie algebra, as well as a dual type-$(-I)$ pre-Lie algebra where $x\succ_{A}y=-y\prec_{A}x$ for all $x,y\in A$.
An apre-perm algebra is a type-$ \begin{pmatrix}
    1 & 1\\
    1 & 2
\end{pmatrix}$ pre-perm algebra, as well as a dual type-$ \begin{pmatrix}
1 & 0\\
1 & -1
\end{pmatrix}$ pre-perm algebra.

There is a one-to-one correspondence between (dual)  \quadco pre-$\mathcal{P}$-algebras and invertible (dual) \quadco $\mathcal{O}$-operators of $\mathcal{P}$-algebras.
Furthermore, when $M$ is nonsingular, we show that   a $\mathcal{P}$-algebra with a nondegenerate bilinear form $\calb$ satisfying the following {\bf \quadco invariant condition}
\begin{equation}\label{eq:M-inv}
    |M|\mathcal{B}(x\circ_{A}y,z)
    =\mathcal{B}(x,b_{1}y\circ_{A}z-a_{1}z\circ_{A}y)
    +\mathcal{B}(y,a_{2}z\circ_{A}x-b_{2}x\circ_{A}z),\ \forall x, y, z\in A,
\end{equation}
%\textcolor{blue}{ZQ: Equations \eqref{eq:M-inv} and \eqref{eq:1299} have duplicate labels.}
gives rise to an invertible dual \quadco
$\mathcal{O}$-operator and hence renders a compatible dual \quadco
pre-$\mathcal{P}$-algebra.

Therefore the relationships among classical splitting of algebras,
$\mathcal O$-operators and bilinear forms have been generalized to
cases of type-$M$. As an example to illustrate the general setup,
we next consider the operad $\mathcal{P}=Leib$ of {\bf Leibniz
algebras} $(A,\circ_A)$ with the defining identity
\begin{equation}\label{eq:Leibniz} x\circ_{A}(y\circ_{A}z)=(x\circ_{A}y)\circ_{A}z+y\circ_{A}(x\circ_{A}z),\;\;\forall x,y,z\in A.
\end{equation}

The notion of Leibniz algebras originated from the work
\cite{Blo1} of Bloh under the name of D-algebras, and they
were rediscovered and studied in \cite{Lod}. In recent
years, Leibniz algebras have found connections with many
areas such as integration deformation quantization, rational
homotopy theory, higher-order differential operators and higher gauge theories \cite{Bor, Dhe,Kot,Str}.

The operad of Leibniz algebras is the duplicator of the operad of
Lie algebras \cite{GK,PBGN2}, and can be induced from averaging operators.

\begin{pro}\label{ex:Lie aver}\cite{Agu2000*}
Let $P:A\rightarrow A$ be an {\bf averaging operator} on a Lie
algebra $(A,[-,-]_{A})$, that is,
    \begin{equation}\label{eq:Ao}
        [P(x),P(y)]_{A}=P([P(x),y]_{A}),\;\forall x,y\in A.
    \end{equation}
Then there is a Leibniz algebra $(A,\circ_{A})$ with the
multiplication $\circ_{A}$ defined by
    \begin{equation}\label{eq:Leibniz from aver op}
        x\circ_{A}y=[P(x),y]_{A},\;\forall x,y\in A.
    \end{equation}
\end{pro}
Furthermore, if $(A,[-,-]_{A},\mathcal{B})$ is a {\bf quadratic
Lie algebra} \cite{MeRe} which is a Lie algebra $(A,[-,-]_{A})$
equipped with a nondegenerate symmetric bilinear form
$\mathcal{B}$ that is invariant in the sense that
\begin{eqnarray}
    \mathcal{B}([x,y]_{A},z)=\mathcal{B}(x,[y,z]_{A}),\;\forall x,y,z\in A,
\end{eqnarray}
then $\mathcal{B}$ becomes {\bf left-invariant} on the induced
Leibniz algebra $(A,\circ_{A})$ given by \eqref{eq:Leibniz from
aver op}, in the sense that
\begin{eqnarray}\label{eq:left-inv}
\mathcal{B}(x\circ_{A}y,z)+\mathcal{B}(y,x\circ_{A}z)=0,\;\forall
x,y,z\in A.
\end{eqnarray}
Even though this left-invariant condition~(\ref{eq:left-inv})
does not belong to the family of type-$M$ invariant conditions
given by (\ref{eq:M-inv}), when $\calb$ is symmetric, the condition is equivalent to the condition
\begin{eqnarray}\label{eq:a-type}
\mathcal{B}(x\circ_{A}y,z)=-\mathcal{B}(y,x\circ_{A}z+z\circ_{A}x)
-\mathcal{B}(x,z\circ_{A}y),\;\forall x,y,z\in A,
\end{eqnarray}
which is precisely the type-$b$ invariant condition given by
(\ref{eq:M-inv}), where
$	b:=
	\begin{pmatrix}
		1 & 0 \\
		-1 & 1
	\end{pmatrix}.
$	
Therefore a Leibniz  algebra $(A,\circ_{A})$ equipped with a nondegenerate bilinear form $\calb$
satisfying (\ref{eq:a-type}) gives an invertible dual type-$b$
$ \mathcal O$-operator and hence a dual
type-$b$ pre-Leibniz algebra $(A,\succ_{A},\prec_{A})$. Moreover,
 due to a representation property
observed in \cite{ST}, $(A,\succ_{A},\prec_{A})$ is a  dual
type-$b$ pre-Leibniz algebra if and only if it is  a \ctpla for $a:=\begin{pmatrix}
        1 & -1  \\
        -1 & 0
    \end{pmatrix}
$.
In addition, when the bilinear form $\calb$
is symmetric (and hence is left-invariant in this case)
 the \ctpla $(A,\succ_{A},\prec_{A})$ obtained this way is special in the
sense that the second multiplication $\prec_{A}$ is
anticommutative. Moreover, we show that an averaging Lie algebra
with an admissible condition gives rise to a \sctpla.

\subsection{An application to a bialgebra structure of Leibniz algebras induced from averaging Lie bialgebras}\label{sec:1.3}\

A bialgebra structure is a vector space equipped with an algebra
structure and a coalgebra structure satisfying some compatible
conditions. For instance, a Lie bialgebra \cite{Cha,Dri} is a
vector space equipped with a Lie algebra structure and a Lie
coalgebra structure that are compatible in the sense of the
$1$-cocycle condition. Lie bialgebras have been in the focus of study due to their close relations with Poisson-Lie
groups and the infinitesimalization of quantum groups. There is a
well-known equivalent characterization of a Lie bialgebra, as
a Manin triple of Lie algebras associated to the nondegenerate
symmetric invariant bilinear form. This Manin triple approach has been adapted to bialgebraic studies for numerous algebras, such as
antisymmetric infinitesimal  bialgebras \cite{Bai2010} and pre-Lie bialgebras \cite{Bai2008}.
Central to this approach is to determine the
suitable invariant conditions for the nondegenerate bilinear forms
on the algebras, which are applied in the definition of Manin
triples. In the case of a Leibniz algebra $(A,\circ_{A})$, Chapoton used
the operad theory in \cite{Chap} to determine its (antisymmetric)
invariant  bilinear form, which was used later in the Manin triple of Leibniz algebras associated to the nondegenerate antisymmetric invariant bilinear form, leading to the notion of a  Leibniz bialgebra \cite{BLST,ST}.

With the recent rapid development of algebras equipped with various linear operators, there have been several studies on the
bialgebra structures for algebras with linear operators, which are equivalently interpreted in terms of
Manin triples of the corresponding algebras with the linear
operators satisfying certain compatible conditions  \cite{Bai2021,LLB,Bai2022,BGS}.
Thus it is natural to consider the induced structures of the ``Lie
bialgebras with averaging operators", that is, extending the
construction of Leibniz algebras given by Proposition~\ref{ex:Lie
    aver} to the context of bialgebras. For this purpose, we first
construct a bialgebra theory for averaging Lie algebras (that is,
Lie algebras with averaging operators), namely averaging Lie
bialgebras, which are equivalently characterized by the introduced
notion of Manin triples of averaging Lie algebras.

Since the (nondegenerate and symmetric) bilinear form on the Leibniz algebra induced from a quadratic Lie algebra with an averaging operator is left-invariant, we naturally expect that the induced bialgebra theory is given by Manin triples of Leibniz
algebras associated to the bilinear forms that satisfy such nondegenerate symmetric left-invariant
condition. The one-to-one correspondence between Leibniz algebras
with nondegenerate symmetric left-invariant bilinear forms and
quadratic \sctplas, as a direct application of the general theory
that we establish in Section \ref{sec:1.1-}, is still available in the form of Manin triples. More precisely, there is a one-to-one
correspondence between Manin triples of Leibniz algebras
associated to the nondegenerate symmetric left-invariant bilinear
forms and Manin triples of \sctplas (associated to the
nondegenerate symmetric invariant bilinear forms). Furthermore,
the latter are interpreted as \sctplbs, leading to a new bialgebra
theory for Leibniz algebras.

In summary, we have reached the conclusion that an averaging Lie
bialgebra gives rise to a \sctplb, lifting the connection in Proposition~\ref{ex:Lie aver} that an
averaging Lie algebra gives rise to a Leibniz algebra to the
bialgebra level.

\subsection{Layout of the paper}\

This paper is organized as follows.

In Section \ref{sec:2}, we introduce the general notion of a
(dual) \quadco pre-$\mathcal{P}$-algebra for any matrix $M\in
M_2(\mathbb{K})$. Then we interpret (dual)  \quadco
pre-$\mathcal{P}$-algebras in terms of (dual) \quadco
$\mathcal{O}$-operators of $\mathcal{P}$-algebras. In particular,
a $\mathcal{P}$-algebra with a nondegenerate bilinear form
satisfying type-$M$ invariant condition gives a dual \quadco
$\mathcal{O}$-operator and hence yields a compatible dual \quadco
pre-$\mathcal{P}$-algebra.

In Section \ref{sec:3}, we show that for matrices $a:=\begin{pmatrix}
	1 & -1  \\
	-1 & 0
\end{pmatrix}
$
and $b:=
\begin{pmatrix}
	1 & 0 \\
	-1 & 1
\end{pmatrix}$, a \ctpla is exactly a
dual type-$b$ pre-Leibniz algebra. We give a detailed study on
special \ctplas. The admissible conditions for averaging Lie
algebras are introduced,  inducing \sctplas.

In Section \ref{sec:5}, the induction of Leibniz algebras from averaging Lie algebras is lifted to the level of bialgebras through the notions of \sctplbs and averaging Lie bialgebras.

\smallskip

\noindent
{\bf Notations.}
Throughout this paper, all the vector
spaces and algebras are finite-dimensional over an algebraically
closed field $\mathbb {K}$ of characteristic zero, although many
results and notions remain valid in the infinite-dimensional case.
Moreover, $\mathcal{P}$ denotes a binary quadratic operad with one operation, and $M$ denotes a matrix \begin{math}
 \begin{pmatrix}
    a_{1} & b_{1}\\
    a_{2} & b_{2}
    \end{pmatrix}
\end{math} in $M_2(\mathbb{K})$.
For a vector space $A$, let
$$\tau:A\otimes A\rightarrow A\otimes A,\quad x\otimes y\mapsto y\otimes x,\;\;\;\forall x,y\in A$$
denote the flip operator. For a multiplication $\circ_A:A\otimes
A\rightarrow A$   on $A$, define linear maps ${
L}_{\circ_A}, {R}_{\circ_A}:A\rightarrow {\rm
End}_{\mathbb K}(A)$ respectively by
\begin{eqnarray*}
{ L}_{\circ_A}(x)y:=x\circ_A y,\;\;
{ R}_{\circ_A}(x)y:=y\circ_A x, \;\;\;\forall x, y\in A.
\end{eqnarray*}

Let $A$ and $V$ be vector spaces. For a linear map
$\rho:A\rightarrow\mathrm{End}_{\mathbb K}(V)$, we define the
linear map  $\rho^{*}:A\rightarrow\mathrm{End}_{\mathbb K}(V^{*})$
by
\begin{eqnarray}
\label{eq:dual}    \langle \rho^{*}(x)u^{*},v\rangle=-\langle
u^{*},\rho(x)v\rangle,\;\forall x\in A, u^{*}\in V^{*},v\in V.
\end{eqnarray}

\section{(Dual) \quadco pre-$\mathcal{P}$-algebras
and (dual) \quadco $\mathcal{O}$-operators}\label{sec:2}\

The notions of (dual) \quadco pre-$\calp$-algebra and (dual)
\quadco $\calo$-operators are introduced. There is a compatible
(dual) \quadco pre-$\calp$-algebra of a $\calp$-algebra if and
only if there is an invertible (dual) \quadco $\calo$-operator of
the $\calp$-algebra. In particular, a $\calp$-algebra with a
nondegenerate bilinear form satisfying certain condition gives
 an invertible dual \quadco $\calo$-operator and hence a
compatible dual \quadco pre-$\calp$-algebra.

\subsection{Type-$M$ pre-$\mathcal{P}$-algebras and \quadco $\mathcal{O}$-operators}\label{subsec:2.1}

\begin{defi}\cite{CBG}\label{defi:971}
    Let  $(A,\circ_{A})$ be a $\mathcal{P}$-algebra and $l_{\circ_{A}},r_{\circ_{A}}:A\rightarrow\mathrm{End}_{\mathbb K}(V)$ be linear maps. If the direct sum $A\oplus V$ of vector spaces is again a $\mathcal{P}$-algebra with the multiplication $\circ_{d}$ given by
    \begin{equation}\label{eq:sd Leibniz} (x+u)\circ_{d}(y+v)=x\circ_{A}y+l_{\circ_{A}}(x)v+r_{\circ_{A}}(y)u,
        \;\forall x,y\in A, u,v\in V,
    \end{equation}
    then we call $(l_{\circ_{A}},r_{\circ_{A}},V)$ a {\bf representation} of the $\mathcal{P}$-algebra $(A,\circ_{A})$.
    We denote the $\mathcal{P}$-algebra structure on $A\oplus V$ by $A\ltimes_{l_{\circ_{A}},r_{\circ_{A}}}V$.
    Two representations $(l _{\circ_{A}},r _{\circ_{A}},V)$ and $(l'_{\circ_{A}},r'_{\circ_{A}},V')$ of $(A,\circ_{A})$ are called \textbf{equivalent} if there exists a linear isomorphism $\phi:V\rightarrow V'$ such that the following equations hold:
    \begin{equation}\label{eq:eq Leibniz rep}
        \phi l_{\circ_{A}}(x)=l'_{\circ_{A}}(x)\phi,\; \phi r_{\circ_{A}}(x)=r'_{\circ_{A}}(x)\phi,\;\forall x\in A.
    \end{equation}
\end{defi}

\begin{ex}
Let $(A,\circ_{A})$ be a $\mathcal{P}$-algebra. Then $( L_{\circ_{A}}, R_{\circ_{A}},A)$ is a representation of $(A,\circ_{A})$, which is called the \textbf{adjoint representation} of $(A,\circ_{A})$.
\end{ex}

\begin{defi}
Let $A$ be a vector space with multiplications $\succ_{A}$,
$\prec_{A}:A\otimes A\rightarrow A$ such that
$(A,\succ_{A},\prec_{A})$ is a $\mathcal{P}$-admissible algebra,
that is, $(A,\circ_{A}=\succ_{A}+\prec_{A})$ is a $\calp$-algebra.
If there exists $M\in M_2(\mathbb{K})$ such that
\repequ is a representation of the $\calp$-algebra
$(A,\circ_{A})$, then we say that $(A,\succ_{A},\prec_{A})$ is a
{\bf \quadco pre-$\calp$-algebra}. Moreover, $(A,\succ_{A},\prec_{A})$ is called {\bf compatible} on $(A,\circ_A)$ and $(A,\circ_A)$ is called the {\bf sub-adjacent $\calp$-algebra of $(A,\succ_{A},\prec_{A})$}.
\end{defi}

\begin{defi}\label{defi:O-ope}
Let $(A,\circ_{A})$ be a $\mathcal{P}$-algebra and $M\in M_2(\mathbb{K})$. Suppose that $\alpha ,\beta :A\rightarrow\mathrm{End}_{\mathbb K}(V)$ are linear maps such that $\big((\alpha ,\beta )M,V \big)$ is a representation of $(A,\circ_{A})$.
If a linear map $T:V\rightarrow A$ satisfies
    \begin{eqnarray}
(Tu)\circ_{A}(Tv)=T\big(\alpha(Tu)v+\beta(Tv)u\big),\;\forall u,v\in V,
\label{eq:dual generic O-ope}
    \end{eqnarray}
    then we say that $T$ is a
    {\bf \quadco $\mathcal{O}$-operator of $(A,\circ_{A})$ associated to $(\alpha ,\beta ,V )$}.
    In particular, a
    \quadco $\mathcal{O}$-operator $T$ is called {\bf strong} if there exists a $\mathcal{P}$-algebra structure on $V$ given by
    \begin{equation}
        u\circ_{V}v=\alpha(Tu)v+\beta(Tv)u,\;\forall
        u,v\in V.
    \label{eq:circ}
\end{equation}
\end{defi}

Recall that the classical notion of an {\bf $\calo$-operator} of a $\calp$-algebra $(A,\circ_{A})$ associated to a representation  $(l_{\circ_{A}},r_{\circ_{A}},V)$ is a linear map
$T:V\rightarrow A$ satisfying
\begin{eqnarray}\label{eq:oop}
    (Tu)\circ_{A}(Tv)=T\big( l_{\circ_{A}}(Tu)v+r_{\circ_{A}}(Tv)u \big),\;\forall u,v\in V,
\end{eqnarray}
which first arose on Lie algebras \cite{Ku} in the study of
the classical Yang-Baxter equation. It is clear
that $T$ is an $\calo$-operator of $(A,\circ_{A})$ associated to
$(l_{\circ_{A}},r_{\circ_{A}},V)$ if and only if $T$ is a
type-\begin{math}
    \begin{pmatrix}
        1 & 0 \\
        0 & 1
    \end{pmatrix}
\end{math}
 $\calo$-operator of $(A,\circ_{A})$ associated to $(l_{\circ_{A}},r_{\circ_{A}},V)$.
 Further examples can be found in Examples\,\ref{e:leib1} and \ref{e:leib2}.

\begin{defi}
    Let $(A,\circ_{A})$ be a $\calp$-algebra and $M\in M_2(\mathbb{K})$
such that $|M|\neq 0$.
        If the following equation holds:
        \begin{eqnarray}\label{eq:870-}
            |M|R(x)\circ_{A}R(y)=R\Big(b_{2}R(x)\circ y-a_{2}y\circ_{A}R(x)+a_{1}x\circ_{A}R(y)-b_{1}R(y)\circ_{A}x\Big),\;\forall x,y\in A,
    \end{eqnarray}
then we say $R$ is a {\bf type-$M$ Rota-Baxter operator of $(A,\circ_{A})$}.
A type-$M$ Rota-Baxter operator $R$ of $(A,\circ_{A})$ is called {\bf strong} if $(A,\star_{A})$ is a $\calp$-algebra, where
\begin{eqnarray}\label{eq:1150}
x\star_{A}y=b_{2}R(x)\circ y-a_{2}y\circ_{A}R(x)+a_{1}x\circ_{A}R(y)-b_{1}R(y)\circ_{A}x, \;\forall x,y\in A.
\end{eqnarray}
\end{defi}

\begin{pro}
    Let $(A,\circ_{A})$ be a $\calp$-algebra, $R:A\rightarrow A$ be a linear map and $M\in M_2(\mathbb{K})$
    such that $|M|\neq 0$.
    Then $R$ is a  type-$M$ Rota-Baxter operator of
$(A,\circ_{A})$ if and only if $R$ is a type-$M$
$\mathcal{O}$-operator of $(A,\circ_{A})$ associated to $(
L_{\succ_{A}}, R_{\prec_{A}},A)$, where
$\succ_{A},\prec_{A}:A\otimes A\rightarrow A$ are multiplications
defined by $( L_{\succ_{A}}, R_{\prec_{A}})=( L_{\circ_{A}},
R_{\circ_{A}})M^{-1}$. In this case, $R$ is strong as a type-$M$
Rota-Baxter operator of $(A,\circ_{A})$ if and only if  $R$ is
strong as a type-$M$ $\mathcal{O}$-operator of  $(A,\circ_{A})$
associated to $( L_{\succ_{A}}, R_{\prec_{A}},A)$, that is,
$(A,\diamond_{A})$ is a $\calp$-algebra with the multiplication
$\diamond_{A}:A\otimes A\rightarrow A$ given by
\begin{eqnarray}
    x\diamond_{A}y=R(x)\succ_{A}y+x\prec_{A}R(y),\;\forall x,y\in A.
\label{eq:862-}
\end{eqnarray}
\end{pro}
\begin{proof}
    Suppose that $R$ is a type-$M$ Rota-Baxter operator of
    $(A,\circ_{A})$. Define two multiplications
$\succ_{A},\prec_{A}:A\otimes A\rightarrow A$ by $( L_{\succ_{A}},
R_{\prec_{A}})=( L_{\circ_{A}}, R_{\circ_{A}})M^{-1}$. That is,
 \begin{eqnarray}\label{eq:1176}
        x\succ_{A}y=|M|^{-1}\big(  b_{2}x\circ_{A}y-a_{2}y\circ_{A}x \big),\;
        x\prec_{A}y=|M|^{-1}\big(  a_{1}x\circ_{A}y-b_{1}y\circ_{A}x \big),\;\forall x,y\in A.
    \end{eqnarray} Moreover, by \eqref{eq:870-}, the following
equation holds
\begin{equation}\label{eq:857-}
R(x)\circ_{A}R(y)=R(R(x)\succ_{A}y+x\prec R(y)),\;\forall x,y\in
A,
\end{equation}
that is, $R$ is a type-$M$ $\calo$-operator of $(A,\circ_{A})$
associated to $( L_{\succ_{A}}, R_{\prec_{A}},A)$. Conversely,
suppose that $( L_{\succ_{A}}, R_{\prec_{A}})=( L_{\circ_{A}},
R_{\circ_{A}})M^{-1}$ and $R$ is a type-$M$ $\calo$-operator of
$(A,\circ_{A})$ associated to $( L_{\succ_{A}}, R_{\prec_{A}},A)$,
that is, \eqref{eq:1176} and \eqref{eq:857-} hold. Then
\eqref{eq:870-} holds. Hence the first half part holds. Moreover,
note that for all $x,y\in A$, we have
\begin{eqnarray*} x\star_{A}y&\overset{\eqref{eq:1150}}{=}&b_{2}R(x)\circ_{A}y-a_{2}y\circ_{A}R(x)+a_{1}x\circ_{A}R(y)-b_{1}R(y)\circ_{A}x\\
    &\overset{\eqref{eq:1176}}{=}&|M|\big(
    R(x)\succ_{A}y+x\prec_{A}R(y)\big)\overset{\eqref{eq:862-}}{=}|M|x\diamond_{A}y.
\end{eqnarray*}
Thus $(A,\star_{A})$ is a $\calp$-algebra if and only if
$(A,\diamond_{A})$ is a $\calp$-algebra. Hence the second half
part holds.
    \end{proof}

\begin{pro}\label{pro:1-}
    Let  $(A,\circ_{A})$ be a $\mathcal{P}$-algebra. Suppose that
    $T:V\rightarrow A$ is a \quadco $\mathcal{O}$-operator of $(A,\circ_{A})$ associated to $(\alpha ,\beta ,V )$. Define the multiplications $\succ_{V}$,
    $\prec_{V}$ on $V$ by
    \begin{eqnarray}
u\succ_{V} v:=\alpha(Tu)v,\; u\prec_{V} v:=\beta(Tv)u,\;\forall u,v\in V.
    \label{eq:P-alg1}
\end{eqnarray}
    Then $(V,\succ_{V},\prec_{V})$ is a \quadco pre-$\mathcal{P}$-algebra if and only if $T$ is strong.
\end{pro}

\begin{proof}
Suppose that $(V,\succ_{V},\prec_{V})$ is a
\quadco pre-$\mathcal{P}$-algebra. Then the sub-adjacent $\mathcal{P}$-algebra is exactly $(V,\circ_{V})$ with $\circ_{V}$ defined by \eqref{eq:circ}. Hence $T$ is strong.

Conversely, suppose that $T$ is a strong \quadco
$\mathcal{O}$-operator of $(A,\circ_{A})$ associated to $(\alpha
,\beta ,V )$. Then $(V,\circ_{V})$ with $\circ_{V}$ defined by \eqref{eq:circ} is a $\mathcal{P}$-algebra.
Let $r$ be a quadratic relation satisfied by $\circ_A$. Then there are $k_i\in \mathbb{K}, 1\leq i\leq 12$ such that $r$ is of the form \cite{LV,Val}
\begin{equation}\label{eq:877}
\begin{split}
r=&k_{1}(x\circ_{A}y)\circ_{A}z+k_{2}(x\circ_{A}z)\circ_{A}y
+k_{3}x\circ_{A}(y\circ_{A}z)+k_{4}x\circ_{A}(z\circ_{A}y)\\
&+k_{5}(y\circ_{A}z)\circ_{A}x+k_{6}(y\circ_{A}x)\circ_{A}z
+k_{7}y\circ_{A}(z\circ_{A}x)+k_{8}y\circ_{A}(x\circ_{A}z)\\
&+k_{9}(z\circ_{A}x)\circ_{A}y+k_{10}(z\circ_{A}y)\circ_{A}x
+k_{11}z\circ_{A}(x\circ_{A}y)+k_{12}z\circ_{A}(y\circ_{A}x)\\=&0,\quad \forall x,y,z\in A.\;\;\;\;
\end{split}
%\label{i-identity}
\end{equation}
Then for any representation $(l_{\circ_{A}},r_{\circ_{A}},V)$ of $(A,\circ_{A})$, the following equations hold:
    \begin{eqnarray}
        &&k_{1}l_{\circ_{A}}(x\circ_{A}y)v
        +k_{2}r_{\circ_{A}}(y)l_{\circ_{A}}(x)v
        +k_{3}l_{\circ_{A}}(x)l_{\circ_{A}}(y)v
        +k_{4}l_{\circ_{A}}(x)r_{\circ_{A}}(y)v\nonumber\\
        &&
        +k_{5}r_{\circ_{A}}(x)l_{\circ_{A}}(y)v
        +k_{6}l_{\circ_{A}}(y\circ_{A}x)v
        +k_{7}l_{\circ_{A}}(y)r_{\circ_{A}}(x)v
        +k_{8}l_{\circ_{A}}(y)l_{\circ_{A}}(x)v\nonumber\\
        &&+k_{9}r_{\circ_{A}}(y)r_{\circ_{A}}(x)v
        +k_{10}r_{\circ_{A}}(x)r_{\circ_{A}}(y)v
        +k_{11}r_{\circ_{A}}(x\circ_{A}y)v
        +k_{12}r_{\circ_{A}}(y\circ_{A}x)v=0,\label{rep1-}\\
        &&k_{1}r_{\circ_{A}}(y)l_{\circ_{A}}(x)v
        +k_{2}l_{\circ_{A}}(x\circ_{A}y)v
        +k_{3}l_{\circ_{A}}(x)r_{\circ_{A}}(y)v
        +k_{4}l_{\circ_{A}}(x)l_{\circ_{A}}(y)v\nonumber\\
        &&+k_{5}r_{\circ_{A}}(x)r_{\circ_{A}}(y)v
        +k_{6}r_{\circ_{A}}(y)r_{\circ_{A}}(x)v\nonumber
        +k_{7}r_{\circ_{A}}(y\circ_{A}x)v
        +k_{8}r_{\circ_{A}}(x\circ_{A}y)v\nonumber\\
        &&+k_{9}l_{\circ_{A}}(y\circ_{A}x)v
        +k_{10}r_{\circ_{A}}(x)l_{\circ_{A}}(y)v
        +k_{11}l_{\circ_{A}}(y)l_{\circ_{A}}(x)v
    +k_{12}l_{\circ_{A}}(y)r_{\circ_{A}}(x)v=0,\label{rep2-}\\
        &&k_{1}r_{\circ_{A}}(x)r_{\circ_{A}}(y)v
        +k_{2}r_{\circ_{A}}(y)r_{\circ_{A}}(x)v
        +k_{3}r_{\circ_{A}}(y\circ_{A}x)v
        +k_{4}r_{\circ_{A}}(x\circ_{A}y)v\nonumber\\
        &&+k_{5}l_{\circ_{A}}(y\circ_{A}x)v
        +k_{6}r_{\circ_{A}}(x)l_{\circ_{A}}(y)v
        +k_{7}l_{\circ_{A}}(y)l_{\circ_{A}}(x)v
        +k_{8}l_{\circ_{A}}(y)r_{\circ_{A}}(x)v\nonumber\\
        &&+k_{9}r_{\circ_{A}}(y)l_{\circ_{A}}(x)v
        +k_{10}l_{\circ_{A}}(x\circ_{A}y)v
        +k_{11}l_{\circ_{A}}(x)r_{\circ_{A}}(y)v
        +k_{12}l_{\circ_{A}}(x)l_{\circ_{A}}(y)v=0,\label{rep3-}
    \end{eqnarray}
for all $x,y\in A,v\in V$.

Define a multiplication on $V\oplus V$ by
    \begin{eqnarray*}
        (u+p)\circ_{d}(v+q)&:=&u\circ_{V}v
        +(a_{1} L_{\succ_{V}}+a_{2} R_{\prec_{V}})(u)q
        +(b_{1} L_{\succ_{V}}+b_{2} R_{\prec_{V}})(v)p \nonumber\\
        &\overset{\eqref{eq:P-alg1}}{=}&u\circ_{V}v
        +\rho(Tu)q+\mu(Tv)p,\;\forall u,v,p,q\in V,
    \end{eqnarray*}
    where
$$\rho:= a_{1}\alpha +a_{2}\beta,\quad \mu:=b_{1}\alpha +b_{2}\beta.$$
Let $u,v,w,p,q,t\in V$. Then we have
\begin{eqnarray*}
&&      \big( (u+p)\circ_{d}(v+q) \big)\circ_{d}(w+t)\\
        &=&\big( u\circ_{V}v+\rho(Tu)q+\mu (Tv)p\big)\circ_{d}(w+t)\\
        &=&(u\circ_{V}v)\circ_{V}w+\rho T(u\circ_{V}v)t
        +\mu(TPW)\rho(Tu)q+\mu(Tw)\mu(Tv)p\\
        &\overset{\eqref{eq:dual generic O-ope},\eqref{eq:circ}}{=}&(u\circ_{V}v)\circ_{V}w
        +\rho(Tu\circ_{A}Tv)t+\mu(Tw)\rho(Tu)q+\mu(Tw)\mu(Tv)p,
    \end{eqnarray*}
    and similarly
    \begin{align*}
        &(u+p)\circ_{d}\big((v+q) \circ_{d}(w+t)\big)\\
        &=u\circ_{V}(v\circ_{V}w)
        +\rho(Tu)\rho(Tv)t+\rho(Tu)\mu(Tw)q+\mu(Tv\circ_{A}Tw)p.
    \end{align*}
By the assumption on $T$, the triple $(\rho=a_{1}\alpha+a_{2}\beta, \mu=b_{1}\alpha+b_{2}\beta,V)$ is a representation of $(A,\circ_{A})$. Thus applying the above two equalities we have
    \begin{eqnarray*}
        && k_{1}\big( (u+p)\circ_{d}(v+q) \big)\circ_{d}(w+t)+k_{2}\big( (u+p)\circ_{d}(w+t) \big)\circ_{d}(v+q)\\
        &&
        +k_{3}(u+p)\circ_{d}\big((v+q) \circ_{d}(w+t)\big)+k_{4}(u+p)\circ_{d}\big( (w+t)\circ_{d}(v+q)\big)\\
        &&+k_{5}\big( (v+q)\circ_{d}(w+t) \big)\circ_{d}(u+p)+k_{6}\big( (v+q)\circ_{d}(u+p) \big)\circ_{d}(w+t)\\
        &&
        +k_{7}(v+q)\circ_{d}\big((w+t) \circ_{d}(u+p)\big)+k_{8}(v+q)\circ_{d}\big( (u+p)\circ_{d}(w+t)\big)\\
        &&+k_{9}\big( (w+t)\circ_{d}(u+p) \big)\circ_{d}(v+q)+k_{10}\big( (w+t)\circ_{d}(v+q) \big)\circ_{d}(u+p)\\
        &&
        +k_{11}(w+t)\circ_{d}\big((u+p) \circ_{d}(v+q)\big)+k_{12}(w+t)\circ_{d}\big( (v+q)\circ_{d}(u+p)\big)\\
        &=&k_{1}(u\circ_{V}v)\circ_{V}w+k_{1}\rho(Tu\circ_{A}Tv)t
        +k_{1}\mu(Tw)\rho(Tu)q+k_{1}\mu(Tw)\mu(Tv)p\\
        && +k_{2}(u\circ_{V}w)\circ_{V}v+k_{2}\mu(Tv)\rho(Tu)t
        +k_{2}\rho(Tu\circ_{A}Tw)q+k_{2}\mu(Tv)\mu(Tw)p\\
        && +k_{3}u\circ_{V}(v\circ_{V}w)+k_{3}\rho(Tu)\rho(Tv)t
        +k_{3}\rho(Tu)\mu(Tw)q+k_{3}\mu(Tv\circ_{A}Tw)p\\
        && +k_{4}u\circ_{V}(w\circ_{V}v)+k_{4}\rho(Tu)\mu(Tv)t
        +k_{4}\rho(Tu)\rho(Tw)q+k_{4}\mu(Tw\circ_{A}Tv)p\\
        && +k_{5}(v\circ_{V}w)\circ_{V}u+k_{5}\mu(Tu)\rho(Tv)t
        +k_{5}\mu(Tu)\mu(Tw)q+k_{5}\rho(Tv\circ_{A}Tw)p\\
        &&+k_{6}(v\circ_{V}u)\circ_{V}w+k_{6}\rho(Tv\circ_{A}Tu)t
        +k_{6}\mu(Tw)\mu(Tu)q+k_{6}\mu(Tw)\rho(Tv)p\\
        && +k_{7}v\circ_{V}(w\circ_{V}u)+k_{7}\rho(Tv)\mu(Tu)t
        +k_{7}\mu(Tw\circ_{A}Tu)q+k_{7}\rho(Tv)\rho(Tw)p\\
        && +k_{8}v\circ_{V}(u\circ_{V}w)+k_{8}\rho(Tv)\rho(Tu)t
        +k_{8}\mu(Tu\circ_{A}Tw)q+k_{8}\rho(Tv)\mu(Tw)p\\
        && +k_{9}(w\circ_{V}u)\circ_{V}v+k_{9}\mu(Tv)\mu(Tu)t
        +k_{9}\rho(Tw\circ_{A}Tu)q+k_{9}\mu(Tv)\rho(Tw)p\\
        && +k_{10}(w\circ_{V}v)\circ_{V}u+k_{10}\mu(Tu)\mu(Tv)t
        +k_{10}\mu(Tu)\rho(Tw)q+k_{10}\rho(Tw\circ_{A}Tv)p\\
       &&+k_{11}w\circ_{V}(u\circ_{V}v)+k_{11}\mu(Tu\circ_{A}Tv)t
        +k_{11}\rho(Tw)\rho(Tu)q+k_{11}\rho(Tw)\mu(Tv)p\\
       &&+k_{12}w\circ_{V}(v\circ_{V}u)+k_{12}\mu(Tv\circ_{A}Tu)t
       +k_{12}\rho(Tw)\mu(Tu)q+k_{12}\rho(Tw)\rho(Tv)p\\
        &{=}&0.
    \end{eqnarray*}
Here the last equality follows from applying each of the four equalities \eqref{eq:877}-\eqref{rep3-} to each of the four columns.
    Thus $(V\oplus V,\circ_{d})$ is also a $\mathcal{P}$-algebra. Hence $\big(( L_{\succ_{V}}, R_{\prec_{V}})M,V  \big)$ is a representation of $(V,\circ_{V})$.
    Therefore $(V,\succ_{V},\prec_{V})$ is a type-$M$ pre-$\mathcal{P}$-algebra.
\end{proof}

Moreover, the $\calo$-operator has the following desired property, which has already been observed in \cite[Thm. 4.4.(b)]{BBGN} for Rota-Baxter operators in the non-relative case.
\begin{pro}\label{pro:notation}
An $\mathcal{O}$-operator in the classical sense is automatically strong.
\end{pro}

\begin{proof}
We follow the same notations as in the proof of
Proposition\,\ref{pro:1-}. Let $T:V\rightarrow A$ be an
$\mathcal{O}$-operator of a $\calp$-algebra $(A,\circ_{A})$
associated to a representation $(l_{\circ_{A}},r_{\circ_{A}},V)$.
    Suppose that the $\mathcal{P}$-algebra given by the binary quadratic operad $\mathcal{P}$ contains \eqref{eq:877} as one of the identities. Then \eqref{rep1-}-\eqref{rep3-} hold.
    Let $u,v,w\in V$. Then we have
    \begin{eqnarray*} (u\circ_{V}v)\circ_{V}w&=&\big(l_{\circ_{A}}(Tu)v+r_{\circ_{A}}(Tv)u\big)\circ_{V}w\\
        &=&l_{\circ_{A}}T\big(l_{\circ_{A}}(Tu)v+r_{\circ_{A}}(Tv)u\big)w
        +r_{\circ_{A}}(Tw)l_{\circ_{A}}(Tu)v+r_{\circ_{A}}(Tw)r_{\circ_{A}}(Tv)u\\
        &\overset{\eqref{eq:oop}}{=}&l_{\circ_{A}}(Tu\circ_{A}Tv)w
        +r_{\circ_{A}}(Tw)l_{\circ_{A}}(Tu)v+r_{\circ_{A}}(Tw)r_{\circ_{A}}(Tv)u,\\
        u\circ_{V}(v\circ_{V}w)&=&u\circ_{V}\big( l_{\circ_{A}}(Tu)v+r_{\circ_{A}}(Tv)u \big)\\
        &=&l_{\circ_{A}}(Tu)l_{\circ_{A}}(Tv)w+l_{\circ_{A}}(Tu)r_{\circ_{A}}(Tw)v
        +r_{\circ_{A}}T\big(l_{\circ_{A}}(Tv)w+r_{\circ_{A}}(Tw)v\big)u\\
        &\overset{\eqref{eq:oop}}{=}&l_{\circ_{A}}(Tu)l_{\circ_{A}}(Tv)w
        +l_{\circ_{A}}(Tu)r_{\circ_{A}}(Tw)v+r_{\circ_{A}}(Tv\circ_{A}Tw)u.
    \end{eqnarray*}
    Thus by \eqref{rep1-}-\eqref{rep3-}, we have
    \begin{eqnarray*}
        &&k_{1}(u\circ_{V}v)\circ_{V}w+k_{2}    (u\circ_{V}w)\circ_{V}v+k_{3}u\circ_{V}(v\circ_{V}w)
        +k_{4}u\circ_{V}(w\circ_{V}v)\\
        &&+k_{5}(v\circ_{V}w)\circ_{V}u+k_{6}   (v\circ_{V}u)\circ_{V}w+k_{7}v\circ_{V}(w\circ_{V}u)
        +k_{8}v\circ_{V}(u\circ_{V}w)\\
        &&+k_{9}(w\circ_{V}u)\circ_{V}v+k_{10}  (w\circ_{V}v)\circ_{V}u+k_{11}w\circ_{V}(u\circ_{V}v)
        +k_{12}w\circ_{V}(v\circ_{V}u)=0.
    \end{eqnarray*}
  Therefore $(V,\circ_{V})$ is a $\calp$-algebra and hence $T$ is strong.
\end{proof}

\begin{pro}\label{pro:strong1-}
 An invertible \quadco $\mathcal{O}$-operator of a $\mathcal{P}$-algebra $(A,\circ_{A})$ is automatically strong.
\end{pro}

\begin{proof}
        Let $T$ be a
        \quadco $\mathcal{O}$-operator of $(A,\circ_{A})$ associated to $(\alpha,\beta,V)$. Then with the notations in the proof of Proposition \ref{pro:1-}, we have
        \begin{eqnarray*}
            T^{-1}(Tu\circ_{A}Tv)
            %=T^{-1}\Big(T\big(\alpha(Tu)v+\beta(Tv)u\big)\Big)
            \overset{\eqref{eq:dual generic O-ope}}{=} \alpha(Tu)v+\beta(Tv)u
            %=u\succ_{V}v+u\prec_{V}v
            =u\circ_{V}v,\;\forall u,v\in V.
        \end{eqnarray*}
        Thus $(V,\circ_{V})$ is a $\mathcal{P}$-algebra which is isomorphic to $(A,\circ_{A})$. Hence $T$ is strong.
\end{proof}

\begin{thm}\label{thm:2-}
Let $(A,\circ_{A})$ be a $\mathcal{P}$-algebra. Then there is a compatible \quadco pre-$\mathcal{P}$-algebra structure $(A,\succ_{A},\prec_{A})$ on $(A,\circ_{A})$
if and only if there is an invertible
\quadco $\mathcal{O}$-operator $T$ of $(A,\circ_{A})$ associated to $(\alpha ,\beta ,V )$ for some linear maps $\alpha ,\beta:A\rightarrow\mathrm{End}_{\mathbb K}(V)$. In this case, the multiplications $\succ_{A}$ and $\prec_{A}$ are defined by
\begin{equation}\label{eq:thm2}
x\succ_{A}y=T\big(\alpha(x)T^{-1}(y)\big),
\;x\prec_{A}y=T\big(\beta(y)T^{-1}(x)\big),\;\forall x,y\in A.
\end{equation}
\end{thm}

\begin{proof}
    Suppose that $T:V\rightarrow A$ is an invertible
    \quadco $\mathcal{O}$-operator of $(A,\circ_{A})$ associated to $(\alpha,\beta,V)$.
    Then there is an induced
    \quadco pre-$\mathcal{P}$-algebra structure $(V,\succ_{V},\prec_{V})$ on $V$ given by \eqref{eq:P-alg1}. The linear isomorphism $T$ gives a
    \quadco pre-$\mathcal{P}$-algebra structure $(A,\succ_{T},\prec_{T})$ on $A$ by
    \begin{eqnarray*} &&x\succ_{T}y:=T\big(T^{-1}(x)\succ_{V}T^{-1}(y)\big)\overset{\eqref{eq:P-alg1}}{=}T\big(\alpha(x)T^{-1}(y)\big):=x\succ_{A}y,\\ &&x\prec_{T}y:=T\big(T^{-1}(x)\prec_{V}T^{-1}(y)\big)\overset{\eqref{eq:P-alg1}}{=}T\big(\beta(y)T^{-1}(x)\big):=x\prec_{A}y,\;\forall x,y\in A.
    \end{eqnarray*}
    Moreover, obviously $x\succ_{A}y+x\prec_{A}y=x\circ_{A}y$ for all $x,y\in A$.
    Hence $(A,\succ_{A},\prec_{A})$ is a compatible
    \quadco pre-$\mathcal{P}$-algebra structure on $(A,\circ_{A})$.

    Conversely, suppose that $(A,\succ_{A},\prec_{A})$ is a
    \quadco pre-$\mathcal{P}$-algebra. Then $T=\mathrm{id}\in\mathrm{Hom}_{\mathbb K}(A,A)=\mathrm{End}_{\mathbb K}(A)$ is an invertible
    \quadco $\mathcal{O}$-operator of $(A,\circ_{A})$ associated to $( L _{\succ_{A}},$
$ R _{\prec_{A}},A)$.
\end{proof}

\subsection{Dual \quadco pre-$\mathcal{P}$-algebras
and dual \quadco $\mathcal{O}$-operators}\

In the following, we shall introduce the notions of dual \quadco pre-$\calp$-algebras, dual \quadco $\calo$-operators and study their relation. Many results on dual \quadco pre-$\calp$-algebras and dual \quadco $\calo$-operators are similar to the  cases in Section \ref{subsec:2.1}, and hence we will omit the detail.

\begin{defi}
Let $(A,\succ_{A},\prec_{A})$ be a $\mathcal{P}$-admissible
algebra such that $(A,\circ_{A}%=\succ_{A}+\prec_{A}
)$ is the
sub-adjacent $\calp$-algebra. If there exists $M\in
M_2(\mathbb{K})$ such that \dualrepequ is a representation of
$(A,\circ_{A})$, then we say $(A,\succ_{A},\prec_{A})$ is a {\bf
dual \quadco pre-$\calp$-algebra}. Moreover, $(A,\succ_{A},\prec_{A})$ is called {\bf compatible} on $(A,\circ_A)$ and $(A,\circ_A)$ is called the {\bf sub-adjacent $\calp$-algebra of $(A,\succ_{A},\prec_{A})$}.
\end{defi}

\begin{defi}\label{defi:dual type O-ope}
Let $(A,\circ_{A})$ be a $\mathcal{P}$-algebra. Suppose that $\alpha ,\beta :A\rightarrow\mathrm{End}_{\mathbb K}(V)$ are linear maps such that $\big((\alpha^*,\beta^*)M,V^*\big)= (a_{1}\alpha^*+a_{2}\beta^*,b_{1}\alpha^*+b_{2}\beta^*,V^*)$ is a representation of $(A,\circ_{A})$.
If there is a linear map $T:V\rightarrow A$ satisfying \eqref{eq:dual generic O-ope},
then we say that $T$ is a dual
{\bf \quadco $\mathcal{O}$-operator of $(A,\circ_{A})$ associated to $(\alpha ,\beta ,V )$}.
In particular, a dual
 \quadco $\mathcal{O}$-operator is called {\bf strong} if there exists a $\mathcal{P}$-algebra structure on $V$ given by \eqref{eq:circ}.
\end{defi}

\begin{pro}\label{pro:1}
Let $(A,\circ_{A})$ be a $\mathcal{P}$-algebra. Suppose that
$T:V\rightarrow A$ is a dual
\quadco $\mathcal{O}$-operator of $(A,\circ_{A})$ associated to $(\alpha ,\beta ,V )$. Define the multiplications $\succ_{V}$,
$\prec_{V}:V\otimes V\rightarrow V$ by \eqref{eq:P-alg1}.
Then $(V,\succ_{V},\prec_{V})$ is a dual
\quadco pre-$\mathcal{P}$-algebra if and only if $T$ is strong. In this case, $T:V\rightarrow A$ is a $\calp$-algebra homomorphism.
\end{pro}

\begin{proof}
It is similar to the proof of Propositions \ref{pro:1-}.
\end{proof}

\begin{pro}\label{pro:strong1}
 An invertible dual \quadco $\mathcal{O}$-operator of a $\mathcal{P}$-algebra $(A,\circ_{A})$ is automatically strong.
\end{pro}

\begin{proof}
It is similar to the proof of Proposition \ref{pro:strong1-}.
\end{proof}

\begin{thm}\label{thm:2}
Let $(A,\circ_{A})$ be a $\mathcal{P}$-algebra. Then there is a
compatible dual \quadco pre-$\mathcal{P}$-algebra structure
$(A,\succ_{A},\prec_{A})$ on $(A,\circ_{A})$ if and only if there
is an invertible dual \quadco $\mathcal{O}$-operator $T$ of
$(A,\circ_{A})$ associated to $(\alpha ,\beta ,V )$ for some
linear maps $\alpha,\beta:A\rightarrow\mathrm{End}_{\mathbb
K}(V)$. In this case, the multiplications $\succ_{A},\prec_{A}$
are defined by \eqref{eq:thm2}.
\end{thm}

\begin{proof}
It is similar to the proof of Theorem \ref{thm:2-}.
\end{proof}

Let $W$ and $V$ be vector spaces, and $\beta:W\rightarrow V$ be
linear. Denote the linear dual
$\beta^{*}:V^{*}\rightarrow W^{*}$ by
\begin{equation*}
\langle\beta^{*}(u^{*}),v\rangle=\langle u^{*},\beta(v)\rangle,\;\forall u^{*}\in V^{*},v\in W.
\end{equation*}
The isomorphism ${\rm Hom}_{\mathbb K}(V\otimes V,\mathbb K)\cong
{\rm Hom}_{\mathbb K}(V, V^*)$ identifies a bilinear form
$\mathcal{B}:V\otimes V\rightarrow \mathbb K$ on $V$ as a
linear map  $\mathcal B^\natural :V\rightarrow V^*$ by
$$\mathcal B(u,v)=\langle \mathcal B^\natural (u),
v\rangle ,\;\;\forall u,v\in V.$$
Then $\mathcal B$ is nondegenerate if and only if $\mathcal B^\natural$ is invertible.

\begin{thm}\label{thm:3}
Let $(A,\circ_{A})$ be a $\mathcal{P}$-algebra and $M\in
M_2(\mathbb{K})$ be nonsingular. %invertible.
Let $\mathcal{B}$ be a
nondegenerate bilinear form on  $(A,\circ_{A})$ such that the
type-$M$ invariant condition \eqref{eq:M-inv} holds, that is,
\begin{eqnarray}\label{eq:1299}
|M|\mathcal{B}(x\circ_{A}y,z)
=\mathcal{B}(x,b_{1}y\circ_{A}z-a_{1}z\circ_{A}y)
+\mathcal{B}(y,a_{2}z\circ_{A}x-b_{2}x\circ_{A}z), \quad \forall x, y, z\in A.
\end{eqnarray}
Then there is a compatible dual
\quadco pre-$\mathcal{P}$-algebra structure $(A,\succ_{A},$
$\prec_{A})$ with multiplications $\succ_{A},\prec_{A}:A\otimes A\rightarrow A$ respectively defined by
\begin{eqnarray}
&&|M|\mathcal{B}(x\succ_{A}y,z)
=\mathcal{B}(y,a_{2}z\circ_{A}x-b_{2}x\circ_{A}z),\label{eq:16}\\
&&|M|\mathcal{B}(x\prec_{A}y,z)
=\mathcal{B}(x,b_{1}y\circ_{A}z-a_{1}z\circ_{A}y).\label{eq:17}
\end{eqnarray}
Moreover, $( L_{\circ_{A}}, R_{\circ_{A}},A)$ and \dualrep are equivalent as representations of $(A,\circ_{A})$.

Conversely, let $(A,\succ_{A},\prec_{A})$ be a  compatible dual
\quadco pre-$\mathcal{P}$-algebra structure on a $\calp$-algebra
$(A,\circ_{A})$ and $|M|\neq 0$. If  $( L_{\circ_{A}},
R_{\circ_{A}},A)$ and \dualrep are equivalent as representations
of $(A,\circ_{A})$, then there exists a nondegenerate bilinear
form $\mathcal{B}$ on $A$ satisfying
\eqref{eq:1299}.
\end{thm}

\begin{proof}
Let $x,y,z\in A$ and $a^{*}=\mathcal{B}^{\natural}(x),
b^{*}=\mathcal{B}^{\natural}(y)$. Then we have {\small
\begin{eqnarray*}
     |M|\langle \mathcal{B}^{\natural}\big( \mathcal{B}^{\natural^{-1}}(a^{*})\circ_{A} \mathcal{B}^{\natural^{-1}}(b^{*}) \big),z\rangle
        & =&|M|\langle \mathcal{B}^{\natural}(x\circ_{A}y),z\rangle =|M|\mathcal{B}(x\circ_{A}y,z)\\
        & =&\mathcal{B}(x,b_{1}y\circ_{A}z-a_{1}z\circ_{A}y)
        +\mathcal{B}(y,a_{2}z\circ_{A}x-b_{2}x\circ_{A}z)\\
        & =&\langle \mathcal{B}^{\natural}(x), b_{1}y\circ_{A}z-a_{1}z\circ_{A}y\rangle
        +\langle \mathcal{B}^{\natural}(y), a_{2}z\circ_{A}x-b_{2}x\circ_{A}z\rangle\\
        & =&\langle (b_{2} L^{*}_{\circ_{A}}-a_{2} R^{*}_{\circ_{A}})
        \big(\mathcal{B}^{\natural^{-1}}(a^{*})\big)b^{*},z\rangle
        +\langle
        (a_{1} R^{*}_{\circ_{A}}-b_{1} L^{*}_{\circ_{A}})
        \big(\mathcal{B}^{\natural^{-1}}(b^{*})\big)a^{*},z\rangle.
    \end{eqnarray*}}
Hence we have
\begin{equation*}
\mathcal{B}^{\natural^{-1}}(a^{*})\circ_{A}\mathcal{B}^{\natural^{-1}}(b^{*})
=\mathcal{B}^{\natural^{-1}} \bigg(\frac{b_{2} L^*_{\circ_{A}}
    -a_{2} R^*_{\circ_{A}}}{|M|}
\big(\mathcal{B}^{\natural^{-1}}(a^{*})\big)b^{*} + \frac{a_{1}
R^*_{\circ_{A}}
    -b_{1} L^*_{\circ_{A}}}{|M|}
\big(\mathcal{B}^{\natural^{-1}}(b^{*})\big)a^{*}\bigg).
\end{equation*}
Define the triple $(\alpha,\beta,V)$ by
$$ (\alpha,\beta,V):= \Big(\frac{b_{2} L^*_{\circ_{A}}
-a_{2} R^*_{\circ_{A}}}{|M|}, \frac{a_{1} R^*_{\circ_{A}} -b_{1}
L^*_{\circ_{A}}}{|M|},A^*\Big).$$ Hence
$\big((\alpha^*,\beta^*)M,V^*\big)=\big(L_{\circ_A},R_{\circ_A}, A\big)$ is the
adjoint representation of $(A,\circ_A)$ and hence
$T=\mathcal{B}^{\natural^{-1}}$ is an invertible dual \quadco
$\mathcal{O}$-operator of $(A,\circ_{A})$ associated to the triple
$(\alpha,\beta,V)$. By Theorem \ref{thm:2}, there are
multiplications $\succ_{A}$ and $\prec_{A}$ defined by
\eqref{eq:thm2} such that $(A,\succ_{A},\prec_{A})$ is a dual
\quadco pre-$\mathcal{P}$-algebra. Explicitly, we have
\begin{equation*}
x\succ_{A}y=\mathcal{B}^{\natural^{-1}} \bigg(\frac{b_{2}
L^{*}_{\circ_{A}} -a_{2} R^{*}_{\circ_{A}}}{|M|}
(x)\mathcal{B}^{\natural}(y)\bigg),\;\;
x\prec_{A}y=\mathcal{B}^{\natural^{-1}} \bigg(\frac{a_{1}
R^{*}_{\circ_{A}} -b_{1} L^{*}_{\circ_{A}}}{|M|}
(y)\mathcal{B}^{\natural}(x)\bigg),
\end{equation*}
and thus
\begin{eqnarray*}
   &&  |M|\mathcal{B}(x\succ_{A}y,z)=\langle
    (b_{2} L^{*}_{\circ_{A}} -a_{2} R^{*}_{\circ_{A}})(x)\mathcal{B}^{\natural}(y) ,z\rangle% =\langle \mathcal{B}^{\natural}(y),a_{2}z\circ_{A}x -b_{2}x\circ_{A}z\rangle\\
    =\mathcal{B}(y,a_{2}z\circ_{A}x-b_{2}x\circ_{A}z),\\
 &&    |M|\mathcal{B}(x\prec_{A}y,z)=\langle
    (a_{1} R^{*}_{\circ_{A}}   -b_{1} L^{*}_{\circ_{A}})(y)\mathcal{B}^{\natural}(x), z\rangle %=\langle \mathcal{B}^{\natural}(x), b_{1}y\circ_{A}z-a_{1}z\circ_{A}y\rangle\\
    =\mathcal{B}(x,b_{1}y\circ_{A}z-a_{1}z\circ_{A}y).
\end{eqnarray*}
Moreover, we have
\begin{eqnarray}
\mathcal{B}(a_{1}x\succ_{A}y+a_{2}y\prec_{A}x,z)=
-\mathcal{B}(y,x\circ_{A}z),\;\;
\mathcal{B}(b_{1}x\succ_{A}y+b_{2}y\prec_{A}x,z)=
-\mathcal{B}(y,z\circ_{A}x),\label{eq:bf41}
\end{eqnarray}
which equivalently gives
\begin{eqnarray}
(\mathcal{B}^{\natural})^{*}\big( L_{\circ_{A}}(x)z \big)= (a_{1}
L^{*}_{\succ_{A}} +a_{2}
R^{*}_{\prec_{A}})(x)(\mathcal{B}^{\natural})^{*}(z), \;\;
(\mathcal{B}^{\natural})^{*}\big( R_{\circ_{A}}(x)z \big)= (b_{1}
L^{*}_{\succ_{A}} +b_{2}
R^{*}_{\prec_{A}})(x)(\mathcal{B}^{\natural})^{*}(z).\label{eq:bfeuq2}
\end{eqnarray}
Therefore, the linear isomorphism
$(\mathcal{B}^{\natural})^{*}:A\rightarrow A^{*}$ gives the
equivalence between  $( L_{\circ_{A}}, R_{\circ_{A}},A)$ and
$\big( (L_{\succ_{A}}^{*}, R_{\prec_{A}}^{*})$ $M,A\big)$
 as representations of $(A,\circ_{A})$.

Conversely, let $\phi:A\rightarrow A^{*}$ be a linear isomorphism
which gives the equivalence between $( L _{\circ_{A}},
R_{\circ_{A}},$ $A)$ and \dualrep as representations of
$(A,\circ_{A})$. Then there is a nondegenerate bilinear form
$\mathcal{B}$ on $A$ defined by
\begin{equation}\mlabel{eq:phi2}
    \mathcal{B}(x,y)=\langle x,\phi(y)\rangle,\;\forall x,y\in A,
\end{equation}
that is, we have $\phi=(\mathcal{B}^{\natural})^{*}$. Hence
 \eqref{eq:bfeuq2} holds. Moreover,
 \eqref{eq:bfeuq2} holds if and only if \eqref{eq:bf41} holds and thus
\eqref{eq:16} and \eqref{eq:17} hold. Therefore we obtain
\eqref{eq:1299}. Hence the conclusion follows.
\end{proof}

\begin{rmk} Our type-$M$ pre-$\calp$-algebras can be further generalized to the case when the corresponding representation is the
linear combination of $L_\succ,R_\succ, L_\prec, R_\prec$ or
$L_\succ^*,R_\succ^*, L_\prec^*, R_\prec^*$.  Such generalizations
contain (dual) \quadco pre-$\mathcal{P}$-algebras as a subclass.
However, when it comes to the relations with analogues of
$\calo$-operators, we are usually concerned with (dual) \quadco
pre-$\mathcal{P}$-algebras only.
\end{rmk}

\section{\Sctplas and their realizations from admissible averaging Lie algebras}\label{sec:3}

In this section, we apply the general framework formulated in
Section \ref{sec:2} to type-$M$ pre-Leibniz algebras, in
particular, type-$a$ pre-Leibniz algebras, or equivalently,
dual type-$b$ pre-Leibniz algebras. Special type-$a$
pre-Leibniz algebras are constructed from admissible averaging Lie
algebras.

\subsection{Type-$M$ pre-Leibniz algebras}\

The following is an explicit form of a special case of the general notions in Definition \ref{defi:971}.

\begin{defi}\cite{ST}
    A \textbf{representation} of a Leibniz algebra $(A,\circ_{A})$ is a triple $(l_{\circ_{A}},r_{\circ_{A}},V)$ consisting of a vector space $V$ and linear maps $l_{\circ_{A}},r_{\circ_{A}}:A\rightarrow\mathrm{End}_{\mathbb
K}(V)$ satisfying the following equations:
\begin{eqnarray} &&l_{\circ_{A}}(x\circ_{A}y)v=l_{\circ_{A}}(x)l_{\circ_{A}}(y)v
-l_{\circ_{A}}(y)l_{\circ_{A}}(x)v,\label{eq:rep1}\\
&&r_{\circ_{A}}(x\circ_{A}y)v=l_{\circ_{A}}(x)r_{\circ_{A}}(y)v
-r_{\circ_{A}}(y)l_{\circ_{A}}(x)v,\label{eq:rep2}\\
&&r_{\circ_{A}}(y)l_{\circ_{A}}(x)v=-r_{\circ_{A}}(y)r_{\circ_{A}}(x)v,\;\forall x,y\in A, v\in V.\label{eq:rep3}
\end{eqnarray}
\end{defi}

Denote a matrix $L=\begin{pmatrix}
    1 & -1 \\
    0 & -1
\end{pmatrix}$. Then we have the following result.

\begin{lem}\label{lem:1518}%\cite{ST}
Let $(l_{\circ_{A}},r_{\circ_{A}},V)$ be a representation of a Leibniz algebra $(A,\circ_{A})$. 
Then $\big( (l^{*}_{\circ_{A}},r^{*}_{\circ_{A}})L,V^{*} \big)$
$=(l^{*}_{\circ_{A}}, -l^{*}_{\circ_{A}}-r^{*}_{\circ_{A}}, V^{*})$ is also a representation of $(A,\circ_{A})$. In particular, $\big( ( L^{*}_{\circ_{A}}, R^{*}_{\circ_{A}})L,V^{*} \big)=( L^{*}_{\circ_{A}}, - L^{*}_{\circ_{A}}- R^{*}_{\circ_{A}}, A^{*})$ is a representation of $(A,\circ_{A})$.
Moreover, if $(l_{\circ_{A}},r_{\circ_{A}},V)$ and $(l'_{\circ_{A}},r'_{\circ_{A}},V')$
are equivalent as representations of $(A,\circ_{A})$, then $\big( (l^{*}_{\circ_{A}},r^{*}_{\circ_{A}})L,V^{*} \big)=(l^{*}_{\circ_{A}}, -l^{*}_{\circ_{A}}-r^{*}_{\circ_{A}}, V^{*})$ and $\big( (l'^{*}_{\circ_{A}},r'^{*}_{\circ_{A}})L,V'^{*} \big)=(l'^{*}_{\circ_{A}}, -l'^{*}_{\circ_{A}}-r'^{*}_{\circ_{A}}, V'^{*})$ are equivalent as representations of $(A,\circ_{A})$.
\end{lem}
\begin{proof}
The first half part follows from \cite{ST}. For the second half
part, note that there exists a linear isomorphism
$\phi:V\rightarrow V' $ satisfying \eqref{eq:eq Leibniz rep}.
 Then we have
 {\small
 \begin{align*}
&l^{*}_{\circ_{A}}(x)\phi^{*}=\big(\phi l_{\circ_{A}}(x)\big)^{*}\overset{\eqref{eq:eq Leibniz rep}}{=}\big(  l'_{\circ_{A}}(x)\phi\big)^{*}=\phi^{*}l'^{*}_{\circ_{A}}(x),\\
&-(l^{*}_{\circ_{A}}+r^{*}_{\circ_{A}})(x)\phi^{*}=-\big(\phi (l_{\circ_{A}}+r_{\circ_{A}})(x)\big)^{*}\overset{\eqref{eq:eq Leibniz rep}}{=}-\big(  (l'_{\circ_{A}}+r'_{\circ_{A}})(x)\phi\big)^{*}
=-\phi^{*}(l'^{*}_{\circ_{A}}+r'^{*}_{\circ_{A}})(x),\;\forall x\in A.
\end{align*}}Hence $(l^{*}_{\circ_{A}}, -l^{*}_{\circ_{A}}-r^{*}_{\circ_{A}}, V^{*})$ and $(l'^{*}_{\circ_{A}}, -l'^{*}_{\circ_{A}}-r'^{*}_{\circ_{A}}, V'^{*})$ are equivalent as representations of $(A,\circ_{A})$.
\end{proof}

\begin{pro}\label{pro:SPA1}
    \begin{enumerate}
        \item\label{1733} Let $A$ be a vector space with the multiplications $\succ_{A}, \prec_{A}:A\otimes A\rightarrow A$.
        Then  $(A,\succ_{A},\prec_{A})$ is a \quadco pre-Leibniz algebra if and only if
        $(A,\succ_{A},\prec_{A})$ is a dual type-$ML$ pre-Leibniz algebra.
        \item\label{1743} Let $(A,\circ_{A})$ be a Leibniz algebra, $\alpha,\beta:A\rightarrow\mathrm{End}_{\mathbb K}(V)$ be linear maps and $T:V\rightarrow A$ be a linear map.
        Then $T$ is a \quadco $\calo$-operator of $(A,\circ_{A})$ if and only if $T$ is a dual type-$ML$ $\calo$-operator of $(A,\circ_{A})$.
    \end{enumerate}
\end{pro}
\begin{proof}
\eqref{1733} Suppose that $(A,\succ_{A},\prec_{A})$ is a \quadco pre-Leibniz algebra. Then $(A,\circ_{A}=\succ_{A}+\prec_{A})$ is a Leibniz algebra, and \repequ is a representation of $(A,\circ_{A})$.
By  Lemma \ref{lem:1518}, $\big( ( L^{*}_{\succ_{A}}, R^{*}_{\prec_{A}})ML,A^{*}
\big)=\big(a_{1} L^{*}_{\succ_{A}}+a_{2} R^{*}_{\prec_{A}},
(-a_{1}-b_{1}) L^{*}_{\succ_{A}}
+(-a_{2}-b_{2}) R^{*}_{\prec_{A}},A^{*}\big)
$ is a representation of $(A,\circ_{A})$.
Hence $(A,\succ_{A},\prec_{A})$ is a dual type-$ML$
  pre-Leibniz algebra. The converse side is obtained similarly.

\eqref{1743} It also follows from Lemma \ref{lem:1518}.
\end{proof}

\begin{ex} \label{e:leib1}
    Any Leibniz algebra $(A,\circ_{A})$ is equivalently characterized as a type-\begin{math}
        \begin{pmatrix}
            0 &-1\\
            1 &1
        \end{pmatrix}
    \end{math} pre-Leibniz algebra, or equivalently, a dual type-\begin{math}
 \begin{pmatrix}
        0 &-1\\
        1 &1
    \end{pmatrix}
L=  \begin{pmatrix}
0 & 1\\
1 &-2
\end{pmatrix}
\end{math} pre-Leibniz algebra $(A,\succ_{A}$, $\prec_{A})$, where
\begin{eqnarray}\label{eq:1742}
x\succ_{A}y=x\circ_{A}y-y\circ_{A}x,\;
x\prec_{A}y=y\circ_{A}x,\;\forall x,y\in A.
\end{eqnarray}
On the other hand, suppose that $\omega$ is a nondegenerate antisymmetric bilinear form on a
Leibniz algebra $(A,\circ_{A})$ which is invariant \cite{Chap} in
the sense of
\begin{equation}\label{eq:568}
 \omega(x,y\circ_{A}z)=\omega(x\circ_{A}z+z\circ_{A}x,y),\;\;\forall x,y,z\in A.
\end{equation}
Then \eqref{eq:1299} holds by taking \begin{math}
    M=  \begin{pmatrix}
        0 & 1\\
        1 &-2
    \end{pmatrix}
\end{math}.
By Theorem \ref{thm:3}, there is an induced compatible dual \quadco pre-Leibniz algebra $(A,\succ_{A},\prec_{A})$ given by
\begin{eqnarray*}
&&  -\omega(x\succ_{A}y,z)=\omega(y,z\circ_{A}x+2x\circ_{A}z)\overset{\eqref{eq:568}}{=}\omega(y\circ_{A}x-x\circ_{A}y,z),\\
&&-\omega(x\prec_{A}y,z)=\omega(x,y\circ_{A}z)\overset{\eqref{eq:568}}{=}-\omega(y\circ_{A}x,z).
\end{eqnarray*}
Therefore, we recover \eqref{eq:1742}.
Moreover, $( L_{\circ_{A}}, R_{\circ_{A}},A)$ and $\big( ( L^{*}_{\succ_{A}}, R^{*}_{\prec_{A}})M, A^{*} \big)=( L^{*}_{\circ_{A}}, - L^{*}_{\circ_{A}}- R^{*}_{\circ_{A}},A^{*})$ are equivalent as representations of  $(A,\circ_{A})$, which recovers \cite[Lemma 4.16]{ST}.
\end{ex}

\begin{ex} \label{e:leib2}
    Recall that a {\bf pre-Leibniz algebra} (originally named as a Leibniz dendriform algebra in \cite{ST}) is a vector space $A$ together with multiplications $\succ_{A},\prec_{A}:A\otimes A\rightarrow A$ satisfying the following identities:
    \begin{eqnarray*}
&&(x\circ_{A}y)\succ_{A}z=x\succ_{A}(y\succ_{A}z)- y\succ_{A}(x\succ_{A}z),\\
&&z\prec_{A}(x\circ_{A}y)= x\succ_{A}(z\prec_{A}y)-(x\succ_{A}z)\prec_{A}y,\\
&&(x\succ_{A}z)\prec_{A}y=-(z\prec_{A}x)\prec_{A}y,\;\forall x,y,z\in A,
    \end{eqnarray*}
where $x\circ_{A}y=x\succ_{A}y+x\prec_{A}y$ for all $x,y\in A$.
In fact, $(A,\succ_{A},\prec_{A})$ is a pre-Leibniz algebra if and only if $(A,\succ_{A},\prec_{A})$ is a
type-$I$ pre-Leibniz algebra for \begin{math}
    I=\begin{pmatrix}
        1 &0 \\
        0 &1
    \end{pmatrix}
\end{math}, as well as a dual
type-$(IL=)L$ pre-Leibniz algebra.
On the other hand, suppose that $(A,\circ_{A},\mathcal{B})$ is a symplectic Leibniz algebra \cite{TXS}, that is, $\mathcal{B}$ is a nondegenerate symmetric bilinear form on a Leibniz algebra $(A,\circ_{A})$ such that
\begin{eqnarray*} \mathcal{B}(z,x\circ_{A}y)=-\mathcal{B}(y,x\circ_{A}z)+\mathcal{B}(x,y\circ_{A}z)+\mathcal{B}(x,z\circ_{A}y),\;\forall x,y,z\in A.
\end{eqnarray*}
Then \eqref{eq:1299} holds by taking \begin{math}
M=L
\end{math}.
By Theorem \ref{thm:3}, there is an induced compatible dual $L$-type pre-Leibniz algebra (or equivalently, pre-Leibniz algebra) $(A,\succ_{A},\prec_{A})$ given by
\begin{eqnarray*}
&&\mathcal{B}(x\succ_{A}y,z)=-\mathcal{B}(y,x\circ_{A}z),\\ &&\mathcal{B}(x\prec_{A}y,z)=\mathcal{B}(x,y\circ_{A}z)+\mathcal{B}(x,z\circ_{A}y),\;\forall x,y,z\in A,
\end{eqnarray*}
which recovers \cite[Proposition 3.2]{TXS}.
Moreover, $( L_{\circ_{A}}, R_{\circ_{A}},A)$ and $\big( ( L^{*}_{\succ_{A}}, R^{*}_{\prec_{A}})L, A^{*} \big)=( L^{*}_{\succ_{A}}, - L^{*}_{\succ_{A}}- R^{*}_{\prec_{A}},A^{*})$ are equivalent as representations of  $(A,\circ_{A})$.
\end{ex}

\subsection{Special \ctplas}\

\begin{pro}\label{pdef:1647}
    Let $A$ be a vector space with multiplications $\succ_{A},\prec_{A}:A\otimes A\rightarrow A$. Then the following statements are equivalent.
    \begin{enumerate}
        \item\label{equ1} $(A,\succ_{A},\prec_{A})$ is a \ctpla, where $a= \begin{pmatrix}
            1 & -1\\
            -1 & 0
        \end{pmatrix}$.
        \item\label{equ2} $(A,\succ_{A},\prec_{A})$ is a dual type-$b$ pre-Leibniz algebra, where $b=aL=\begin{pmatrix}
            1 & 0\\
            -1 & 1
        \end{pmatrix}$.
    \item\label{equ3} $(A,\succ_{A},\prec_{A})$ is a Leibniz-admissible algebra and the following equations hold:
    \begin{eqnarray}
        &&(x\circ_{A}y)\bullet_{A}z=x\bullet_{A}(y\bullet_{A}z)-y\bullet_{A}(x\bullet_{A}z),\label{eq:gppa2,2}\\
        &&z\prec_{A}(x\circ_{A}y)=x\bullet_{A}(z\prec_{A}y)-(x\bullet_{A}z)\prec_{A}y,\label{eq:gppa3,2}\\
        &&x\bullet_{A}(z\prec_{A}y)=-(z\prec_{A}y)\prec_{A}x,\;\forall x,y,z\in A,\label{eq:gppa4,2}
    \end{eqnarray}
where \begin{equation}\label{eq:Dpp2}
x\circ_{A}y:=x\succ_{A}y+x\prec_{A}y,\;\;    x\bullet_{A}y:=x\succ_{A}y-y\prec_{A}x,\;\forall x,y\in A.
\end{equation}
    \end{enumerate}
Moreover, in this case both $( L^{*}_{\bullet_{A}},
 R^{*}_{\prec_{A}},A^{*}
)$ and $( L_{\bullet_{A}},- L_{\succ_{A}},A)$ are representations of $(A,\circ_{A})$.
\end{pro}
\begin{proof}
\eqref{equ1}$\Longleftrightarrow$\eqref{equ2}   It follows from Proposition \ref{pro:SPA1} by taking $a_{1}=1, a_{2}=b_{1}=-1, b_{2}=0$.

\eqref{equ2}$\Longleftrightarrow$\eqref{equ3} By \eqref{eq:Dpp2}, we have $ L_{\bullet_{A}}= L_{\succ_{A}}- R_{\prec_{A}}$.
Then for all $x,y,z\in A, a^{*}\in A^{*}$, we have
\begin{eqnarray*}
&&\langle  L^{*}_{\bullet_{A}}(x\circ_{A}y)a^{*},z\rangle=-\langle a^{*},(x\circ_{A}y)\bullet_{A}z\rangle,\\
&&
\langle  L^{*}_{\bullet_{A}}(x) L^{*}_{\bullet_{A}}(y)a^{*},z\rangle=\langle a^{*}, y\bullet_{A}(x\bullet_{A}z)\rangle,\\
&& \langle L^{*}_{\bullet_{A}}(y) L^{*}_{\bullet_{A}}(x)a^{*},z\rangle=\langle a^{*}, x\bullet_{A}(y\bullet_{A}z)\rangle.
\end{eqnarray*}
Hence $( L^{*}_{\bullet_{A}}= L^{*}_{\succ_{A}}- R^{*}_{\prec_{A}},
 R^{*}_{\prec_{A}},A^{*}
)$ satisfies \eqref{eq:rep1}
if and only if \eqref{eq:gppa2,2} holds.
Similarly, $( L^{*}_{\bullet_{A}},
 R^{*}_{\prec_{A}},A^{*}
)$ satisfies \eqref{eq:rep2} and \eqref{eq:rep3}  if and only if \eqref{eq:gppa3,2} and \eqref{eq:gppa4,2} hold respectively.
In conclusion, $( L^{*}_{\bullet_{A}},
 R^{*}_{\prec_{A}},A^{*}
)$ is a representation of $(A,\circ_{A})$ if and only if  \eqref{eq:gppa2,2}-\eqref{eq:gppa4,2} hold.
\end{proof}

\begin{cor}\label{cor:420}
    Let $\mathcal{B}$ be a nondegenerate bilinear form on a Leibniz algebra $(A,\circ_{A})$ satisfying
    \begin{equation}\label{eq:left inv1}        \mathcal{B}(x\circ_{A}y,z)=-\mathcal{B}(y,x\circ_{A}z+z\circ_{A}x)
        -\mathcal{B}(x,z\circ_{A}y),\;\forall x,y,z\in A.
    \end{equation}
    Then there is a compatible  \ctpla structure $(A,\succ_{A},\prec_{A})$ with  multiplications $\succ_{A},\prec_{A}:A\otimes
    A\rightarrow A$ defined respectively by
    \begin{eqnarray}
        &&\mathcal{B}(x\succ_{A}y,z)=-\mathcal{B}(y,x\circ_{A}z+z\circ_{A}x),\label{eq:cor3}\\
        &&\mathcal{B}(x\prec_{A}y,z)=-\mathcal{B}(x,z\circ_{A}y).\label{eq:cor4}
    \end{eqnarray}
    Moreover, $( L_{\circ_{A}}, R_{\circ_{A}},A)$ and $( L^{*}_{\bullet_{A}}$, $R^{*}_{\prec_{A}},A^{*})$ are equivalent as representations of $(A,\circ_{A})$, and $( L_{\bullet_{A}}$, $- L_{\succ_{A}},$
$A)$ and  $( L^{*}_{\circ_{A}}$,
$- L^{*}_{\circ_{A}}
    - R^{*}_{\circ_{A}},A^{*})$ are equivalent as representations of $(A,\circ_{A})$.

    Conversely,
    let $(A,\succ_{A},\prec_{A})$ be a compatible \ctpla structure on a Leibniz algebra $(A,\circ_{A})$. If  $( L _{\circ_{A}}, R _{\circ_{A}},A)$ and $(
     L^{*}_{\bullet_{A}}, R^{*}_{\prec_{A}},A^{*})$
    are equivalent as representations of $(A,\circ_{A})$, then there exists a nondegenerate bilinear form $\mathcal{B}$ on $ A$ satisfying \eqref{eq:left inv1}.
\end{cor}

\begin{proof}
This result follows from  Theorem \ref{thm:3} by taking $M=b$,
Proposition \ref{pdef:1647}
 and Lemma \ref{lem:1518} by taking $(l_{\circ_{A}},r_{\circ_{A}},V)=( L_{\circ_{A}}, R_{\circ_{A}},A)$ and $(l'_{\circ_{A}},r'_{\circ_{A}},V')=( L^{*}_{\bullet_{A}}, R^{*}_{\prec_{A}},A^{*}).$
\end{proof}

In the following, we shall study a special subclass of \ctplas.

\begin{defi}
A \ctpla $(A,\succ_{A},\prec_{A})$ is called {\bf special} if the multiplication $\prec_{A}$ is anticommutative, that is,  $x\prec_{A}y=-y\prec_{A}x$ for all $x,y\in A$.
\end{defi}

\begin{cor}\label{cor:SDPP}
    Let $A$ be a vector space with multiplications $\succ_{A}
    ,\prec_{A}:A\otimes A\rightarrow A$.
    Define a multiplication $\circ_{A} :A\otimes A\rightarrow A$ by
    \eqref{eq:Dpp2}.
    Then the following statements are equivalent:
    \begin{enumerate}
        \item\label{s1} $(A,\succ_{A},\prec_{A})$ is a \sctpla.
        \item\label{s4} The multiplication $\prec_{A}$ is anticommutative,  $(A,\circ_{A})$ is a Leibniz algebra and the following equation holds:
    \begin{equation}\label{eq:SDPP}
            x\circ_{A}(y\prec_{A}z)=(x\circ_{A}y)\prec_{A}z
+y\prec_{A}(x\circ_{A}z)=x\prec_{A}(y\prec_{A}z),\;\forall x,y,z\in A.
        \end{equation}
    \end{enumerate}
Moreover, in this case both $( L^{*}_{\circ_{A}},
 R^{*}_{\prec_{A}},A^{*}
)$ and $( L_{\circ_{A}},- L_{\succ_{A}},A)$ are representations of $(A,\circ_{A})$.
\end{cor}

\begin{proof}
It is clear that $\prec_{A}$ is anticommutative if and only if $\circ_{A}=\bullet_{A}$.
Hence the conclusion follows from Proposition \ref{pdef:1647}.
\end{proof}

\begin{rmk}
The equation \eqref{eq:SDPP} has been observed in \cite[Proposition 1.6]{BeHi}.
From now on, we shall use Corollary \ref{cor:SDPP} \eqref{s4} as an explicit definition of a \sctpla.
\end{rmk}

\begin{cor}\label{cor:420-}
    Let $\mathcal{B}$ be a nondegenerate bilinear form on a Leibniz algebra $(A,\circ_{A})$ satisfying
        \begin{eqnarray}\label{eq:1825} \mathcal{B}(x\circ_{A}y,z)=-\mathcal{B}(y,x\circ_{A}z)=-\mathcal{B}(x\circ_{A}z,y),\;\forall x,y,z\in A.
    \end{eqnarray}
    Then there is a compatible  \sctpla structure $(A,\succ_{A},\prec_{A})$ with  multiplications $\succ_{A},\prec_{A}:A\otimes
    A\rightarrow A$ respectively defined by \eqref{eq:cor3} and \eqref{eq:cor4}.
    Moreover, $( L_{\circ_{A}}, R_{\circ_{A}},A)$ and $( L^{*}_{\circ_{A}}, R^{*}_{\prec_{A}},A^{*})$ are equivalent as representations of $(A,\circ_{A})$, and $( L_{\circ_{A}},$
$- L_{\succ_{A}},A)$ and  $( L^{*}_{\circ_{A}},- L^{*}_{\circ_{A}}
- R^{*}_{\circ_{A}},A^{*})$ are equivalent as representations of $(A,\circ_{A})$.

Conversely, let $(A,\succ_{A},\prec_{A})$ be a compatible \sctpla structure on a Leibniz algebra $(A,\circ_{A})$. If  $( L _{\circ_{A}}, R _{\circ_{A}},A)$ and $(L^{*}_{\circ_{A}}, R^{*}_{\prec_{A}},A^{*})$
are equivalent as representations of $(A,\circ_{A})$, then there exists a nondegenerate bilinear form $\mathcal{B}$ on $ A$ satisfying \eqref{eq:1825}.
\end{cor}

\begin{proof}
    It is clear that \eqref{eq:1825} gives rise to \eqref{eq:left inv1}.
    By Corollary \ref{cor:420}, there is a compatible  \ctpla structure $(A,\succ_{A},\prec_{A})$ with  multiplications $\succ_{A},\prec_{A}:A\otimes
    A\rightarrow A$ respectively defined by \eqref{eq:cor3} and \eqref{eq:cor4}.
    Moreover, we have
    $$\mathcal{B}(x\prec_{A}y,z)=-\mathcal{B}(x,z\circ_{A}y)
    =\mathcal{B}(y, z\circ_{A}x)=-\mathcal{B}(y\prec_{A}x,z),\;\forall x,y,z\in A.$$ By
    the nondegeneracy
    of $\mathcal{B}$, we have $
    x\prec_{A}y=-y\prec_{A}x $ for all $x,y\in A$ and
    hence $(A,\succ_{A}$, $\prec_{A})$ is a \sctpla. By
    Corollary \ref{cor:420},
$( L_{\circ_{A}}, R_{\circ_{A}},A)$ and $( L^{*}_{\bullet_{A}}, R^{*}_{\prec_{A}},A^{*})
=( L^{*}_{\circ_{A}},- L^{*}_{\prec_{A}},A^{*})$ are equivalent as representations of $(A,\circ_{A})$, and $( L_{\bullet_{A}}= L_{\circ_{A}},- L_{\succ_{A}},A)$ and  $( L^{*}_{\circ_{A}},- L^{*}_{\circ_{A}}
$
$    - R^{*}_{\circ_{A}},A^{*})$ are equivalent as representations of $(A,\circ_{A})$. The converse side is obtained similarly.
\end{proof}

\begin{defi}
A bilinear form $\calb$ on a Leibniz algebra $(A,\circ_{A})$ is called
{\bf left-invariant}
if
\begin{eqnarray}\label{eq:li}
\mathcal{B}(x\circ_{A}y,z)+\mathcal{B}(y,x\circ_{A}z)=0,\;\forall x,y,z\in A.
\end{eqnarray}
\end{defi}

\begin{lem}\label{lem:transfer111} Suppose that $\mathcal{B}$
    is a symmetric bilinear form on a Leibniz algebra
    $(A,\circ_{A})$. Then $\mathcal{B}$ is left-invariant
    if and only if \eqref{eq:left inv1} holds. Moreover, in this case, \eqref{eq:1825} holds automatically.
\end{lem}
\begin{proof}
Suppose that \eqref{eq:left inv1} holds, which is rewritten as
    \begin{equation}\label{eq:li2}
        \mathcal{B}(x\circ_{A}y,z)+\mathcal{B}(x\circ_{A}z,y)=
        -\mathcal{B}(z\circ_{A}x,y)-\mathcal{B}(z\circ_{A}y,x),\;\forall x,y,z\in A.
    \end{equation}
    Observing the LHS of \eqref{eq:li2} is symmetric in $y$ and $z$, we have
    \begin{equation} \label{eq:li3}
        -\mathcal{B}(z\circ_{A}x,y)-\mathcal{B}(z\circ_{A}y,x)=
        -\mathcal{B}(y\circ_{A}x,z)-\mathcal{B}(y\circ_{A}z,x),\;\forall x,y,z\in A.
    \end{equation}
    Observing the RHS of \eqref{eq:li2} is symmetric in $x$ and $y$, we have
    \begin{equation}\label{eq:li4}
        \mathcal{B}(x\circ_{A}y,z)+\mathcal{B}(x\circ_{A}z,y)=
        \mathcal{B}(y\circ_{A}x,z)+\mathcal{B}(y\circ_{A}z,x),\;\forall x,y,z\in A.
    \end{equation}
    Combining \eqref{eq:li2}-\eqref{eq:li4} together, we obtain \eqref{eq:li}.
    The rest is obvious.
\end{proof}

\begin{cor}\label{cor:quadratic1}
    Let $\mathcal{B}$ be a nondegenerate symmetric left-invariant bilinear form on a Leibniz
    algebra $(A,\circ_{A})$. Then there is a
    compatible \sctpla structure
    $(A,\succ_{A}$, $\prec_{A})$ with multiplications
    $\succ_{A},\prec_{A}:A\otimes A\rightarrow A$
    given by \eqref{eq:cor3} and \eqref{eq:cor4}.
    Moreover,  $( L_{\circ_{A}}, R_{\circ_{A}},A)$ and $(   L^{*}_{\circ_{A}},- L^{*}_{\prec_{A}},A^{*})$ are equivalent as representations of $(A,\circ_{A})$,     and $(  L_{\circ_{A}},- L_{\succ_{A}},A)$ and  $( L^{*}_{\circ_{A}},- L^{*}_{\circ_{A}}
    - R^{*}_{\circ_{A}},A^{*})$ are equivalent as representations of $(A,\circ_{A})$.
\end{cor}
\begin{proof} It follows from Corollary \ref{cor:420} and Lemma \ref{lem:transfer111}.
\end{proof}

In the following, we will introduce some 
constructions of Leibniz algebras with nondegenerate symmetric left-invariant bilinear forms and the induced \sctplas.

\begin{pro}\label{pro:1926}
    Let $(A,[-,-]_{A} )$ be a   Lie algebra with an averaging operator $P:A\rightarrow A$ and the induced Leibniz algebra $(A,\circ_{A})$ be given by ~\eqref{eq:Leibniz from aver op}.
Suppose that $\mathcal{B}$ is a nondegenerate symmetric invariant
bilinear form on $(A,[-,-]_{A} )$.
Then $\mathcal{B}$ is left-invariant on $(A,\circ_{A})$.
Moreover, suppose that $\widehat{P}$ is the adjoint map of $P$ associated to $\mathcal{B}$, that is,
\begin{equation*} \mathcal{B}\big(P(x),y\big)=\mathcal{B}\big(x,\widehat{P}(y)\big),\;\forall x,y\in A.
\end{equation*}
Then there is a
compatible  \sctpla structure $(A,\succ_{A},\prec_{A})$ on
$(A,\circ_A)$ with multiplications $\succ_A,\prec_A:A\otimes A\rightarrow A$ respectively defined by
\begin{equation}\label{eq:90}
    x\succ_{A}y=[P(x),y]_{A}-\widehat{P}([x,y]_{A}),\; x\prec_{A}y=\widehat{P}([x,y]_{A}),\;\forall x,y\in A.
\end{equation}
\end{pro}
\begin{proof}
   Let $x,y,z\in A$. Then  we have
    \begin{equation*}
        \mathcal{B}(x\circ_{A}y,z)=\mathcal{B}([P(x),y]_{A},z)=
        -\mathcal{B}(y,[P(x),z]_{A})=-\mathcal{B}(y,x\circ_{A}z).
    \end{equation*}
Hence $\mathcal B$ is a
nondegenerate symmetric left-invariant bilinear form on
$(A,\circ_A)$. By Corollary \ref{cor:quadratic1}, there is a
compatible  \sctpla $(A,\succ_{A},\prec_{A})$ of $(A,\circ_{A})$, where
\begin{eqnarray*}
    &&\mathcal{B}(x\prec_{A}y,z)=\mathcal{B}(z\circ_{A}x,y) =\mathcal{B}([P(z),x]_{A},y)=\mathcal{B}\big(\widehat{P}([x,y]_{A}),z\big),\;\forall x,y,z\in A.\label{eq:hatP1}
\end{eqnarray*}
Hence $x\prec_{A}y=\widehat{P}([x,y]_{A})$ and thus
$x\succ_{A}y=x\circ_{A}y-x\prec_{A}y=[P(x),y]_{A}
-\widehat{P}([x,y]_{A})$. 
\end{proof}

\begin{ex}\label{pre-ex:2.3}
    Let $(A,[-,-]_{A})$ be the $3$-dimensional simple Lie algebra
    $\frak s\frak l(2,\mathbb{K})$ with a basis $$\left\{
    x=\left(\begin{matrix} 0&1\cr
        0&0\cr\end{matrix}\right),h=\left(\begin{matrix} 1&0\cr
        0&-1\cr\end{matrix}\right),y=\left(\begin{matrix} 0&0\cr
        1&0\cr\end{matrix}\right)\right\}$$ and with the multiplication
    \begin{equation}
        [h,x]_{A}=2x,\; [h,y]_{A}=-2y,\;[x,y]_{A}=h.
    \end{equation}
    Define a linear map $P:A\rightarrow A$ by
    \begin{equation}\label{eq:ex4}
        P(x)=P(h)=2x+4y+2h,\;P(y)=x+2y+h.
    \end{equation}
    Then $P$ is an averaging operator on $(A,[-,-]_{A})$.
    Thus there is an induced Leibniz algebra $(A,\circ_{A})$   with the following nonzero products:
    \begin{eqnarray*}
&x\circ_{A}x=h\circ_{A}x=4x-4h,
        x\circ_{A}y=h\circ_{A}y=2h-4y,
        x\circ_{A}h=h\circ_{A}h=8y-4x, &\\
&       y\circ_{A}x=2x-2h,\;y\circ_{A}y=h-2y,\;y\circ_{A}h=4y-2x.&
    \end{eqnarray*}
    Define a bilinear form $\mathcal{B}$ on $A$ by
    \begin{eqnarray}
        \mathcal{B}(a,b)=\mathrm{Tr}(ab),\;\forall a,b\in A.
    \end{eqnarray}
    Then $\mathcal{B}$ is a nondegenerate symmetric invariant bilinear form on $(A,[-,-]_{A})$ whose nonzero values are given by
    \begin{equation}\label{eq:10.0}
        \mathcal{B}(x,y)=1,\; \mathcal{B}(h,h)=2.
    \end{equation}
    Moreover, $\mathcal{B}$ is left-invariant on $(A,\circ_{A})$.
    Note that
        \begin{equation*}
    \widehat{P} (x)=\widehat{P}(h)=2x+4y+2h,\;\widehat{P}(y)=x+2y+h,
    \end{equation*}
which coincides with $P$. Therefore there exists a compatible
\sctpla structure $(A,\succ_{A},\prec_{A})$ of $(A,\circ_{A})$
with $\succ_{A},\prec_{A}$ defined by \eqref{eq:90}, whose nonzero
products are
\begin{align*}
    &x\succ_{A}x=4x-4h,&\;&x\succ_{A}y=-2x-8y,&\;
    &x\succ_{A}h=16y+4h,\\
    &y\succ_{A}x=4x+4y,&\;&y\succ_{A}y=h-2y,&\;&y\succ_{A}h=-4x-2h,\\
    &h\succ_{A}x=-8y-8h,&\;&h\succ_{A}y=2x+4h,&\;&h\succ_{A}h=8y-4x,\\
    &x\prec_{A}y=2x+4y+2h,&\;&
    x\prec_{A}h=-4x-8y-4h,&\;&y\prec_{A}h=2x+4y+2h.
\end{align*}
\end{ex}

Now we introduce the notion of representations of \sctplas.

\begin{defi}
    Let $(A,\succ_{A},\prec_{A})$ be a \sctpla and $(A,\circ_{A})$ be
    the sub-adjacent Leibniz algebra. Let $V$ be a vector space and
    $l_{\succ_{A}},r_{\succ_{A}},l_{\prec_{A}}:A\rightarrow\mathrm{End}_{\mathbb
        K}(V)$ be linear maps. Set
    \begin{equation}\label{eq:sum linear}
        l_{\circ_{A}}:=l_{\succ_{A}}+l_{\prec_{A}},\;\;
        r_{\circ_{A}}:=r_{\succ_{A}}-l_{\prec_{A}}.
    \end{equation}
    If $(l_{\circ_{A}},r_{\circ_{A}},V)$ is a representation of $(A,\circ_{A})$ and the following equations hold:
        \begin{eqnarray}
            &&l_{\circ_{A}}(x)l_{\prec_{A}}(y)v
            =l_{\prec_{A}}(x\circ_{A}y)v
            +l_{\prec_{A}}(y)l_{\circ_{A}}(x)v
            =l_{\prec_{A}}(x)l_{\prec_{A}}(y)v,\label{eq:sdpp rep1}\\
            &&r_{\circ_{A}}(x\prec_{A} y)v=
            l_{\prec_{A}}(x)r_{\circ_{A}}(y)v-l_{\prec_{A}}(y)r_{\circ_{A}}(x)v
            =-l_{\prec_{A}}(x\prec_{A}y)v,\;\forall x,y\in A, v\in V,
    \end{eqnarray}then we say that  $(l_{\succ_{A}},r_{\succ_{A}},l_{\prec_{A}},V)$ is a {\bf representation} of $(A,\succ_{A},\prec_{A})$.
\end{defi}

\begin{pro}\label{pro:rep}
    Let $(A,\succ_{A},\prec_{A})$ be a \sctpla and $(A,\circ_{A})$ be
    the sub-adjacent Leibniz algebra. If
    $(l_{\succ_{A}},r_{\succ_{A}},l_{\prec_{A}},V)$ is a
    representation of $(A,\succ_{A},\prec_{A})$, then
    $$(l^{*}_{\circ_{A}}+r^{*}_{\circ_{A}},-r^{*}_{\succ_{A}},
    l^{*}_{\prec_{A}}-r^{*}_{\succ_{A}},V^{*})
    =(l^{*}_{\succ_{A}}+r^{*}_{\succ_{A}},
    -r^{*}_{\succ_{A}},-r^{*}_{\circ_{A}},V^{*})$$
    is also a representation of $(A,\succ_{A},\prec_{A})$.
    In particular, $( L^{*}_{\succ_{A}}+ R^{*}_{\succ_{A}},
- R^{*}_{\succ_{A}},- R^{*}_{\circ_{A}},A^{*})$ is a representation of $(A,\succ_{A},\prec_{A})$, which is called the {\bf coadjoint representation} of $(A,\succ_{A},$
$\prec_{A})$.
\end{pro}

\begin{proof}
    It is clear that $l^{*}_{\circ_{A}}$ satisfies \eqref{eq:rep1}.
    Let $x,y\in A,u^{*}\in V^{*},v\in V$. Then we have
        \begin{eqnarray*}
            &&\langle -l^{*}_{\prec_{A}}(x\circ_{A}y)u^{*},v\rangle =\langle u^{*
            },l _{\prec_{A}}(x\circ_{A}y)v\rangle ,\;
            \langle-l_{\circ_{A}}^{* }(x)l_{\prec_{A}}^{* }(y)u^{*
            },v\rangle=-\langle u^{* },l_{\prec_{A}}(y) l_{\circ_{A}}(x)v\rangle, \\
            &&\langle-l_{\prec_{A}}^{* }(y)l_{\circ_{A}}^{* }(x)u^{* },v\rangle
            =-\langle u^{* },l_{\circ_{A}}(x)l_{\prec_{A}}(y)v\rangle,\;
            \langle-l_{\prec_{A}}^{* }(y)l_{\prec_{A}}^{* }(x)u^{* },v\rangle=-\langle u^{*},l_{\prec_{A}}(x)
            l_{\prec_{A}}(y)v\rangle.
    \end{eqnarray*}Hence by \eqref{eq:sdpp rep1}, we have
        \begin{equation*}
            -l^{*}_{\prec_{A}}(x\circ_{A}y)u^{*}=
            -l_{\circ_{A}}^{* }(x)l_{\prec_{A}}^{* }(y)u^{*}
            +l_{\prec_{A}}^{* }(y)l_{\circ_{A}}^{* }(x)u^{*},\;
            -l_{\prec_{A}}^{* }(y)l_{\circ_{A}}^{* }(x)u^{*}=
            -l_{\prec_{A}}^{* }(y)l_{\prec_{A}}^{*}(x)u^{*}.
    \end{equation*}
    Thus
$$ (l^{*}_{\circ_{A}},-l^{*}_{\prec_{A}},V^{*})
    =(l^{*}_{\succ_{A}}+r^{*}_{\succ_{A}}
    -r^{*}_{\circ_{A}},-r^{*}_{\succ_{A}}-l^{*}_{\prec_{A}}
    +r^{*}_{\succ_{A}},V^{*})$$
    is a representation of $(A,\circ_{A})$.
    Similarly we have
    \begin{eqnarray*}
        &&-l_{\circ_{A}}^{* }(x)r_{\circ_{A}}^{* }(y)u^{* }=-r_{\circ_{A}}^{* }(x\circ_{A}y)u^{*}
        -r_{\circ_{A}}^{* }(y)l_{\circ_{A}}^{* }(x)u^{* }=r_{\circ_{A}}^{* }(x)r_{\circ_{A}}^{* }(y)u^{* }, \\
        &&-l_{\prec_{A}}^{*}(x\prec_{A}y)u^{*}=
        r_{\circ_{A}}^{* }(x)l_{\prec_{A}}^{* }(y)u^{* }-r_{\circ_{A}}^{* }(y)l_{\prec_{A}}^{* }(x)u^{* }=r_{\circ_{A}}^{* }(x\prec_{A}y)u^{*}.
    \end{eqnarray*}
    Thus $(l^{*}_{\succ_{A}}+r^{*}_{\succ_{A}},
    -r^{*}_{\succ_{A}},-r^{*}_{\circ_{A}},V^{*})$ is a representation of $(A,\succ_{A},\prec_{A})$.
\end{proof}

\begin{pro}\label{pro:336}
    Let $(A,\succ_{A},\prec_{A})$ be a \sctpla and $(A,\circ_{A})$ be
    the sub-adjacent Leibniz algebra. Then  $( L_{\succ_{A}}, R_{\succ_{A}},
     L_{\prec_{A}},A)$
    and
    $( L^{*}_{\succ_{A}}+ R^{*}_{\succ_{A}} ,- R^{*}_{\succ_{A}},- R^{*}_{\circ_{A}},$
$A^{*})$
    are equivalent as representations of $(A,\succ_{A},\prec_{A})$
    if and only if
    there is a nondegenerate bilinear form $\mathcal{B}$ on $A$ satisfying \eqref{eq:cor3},  \eqref{eq:cor4} and
    \begin{equation}\label{eq:cor3.37}
        \mathcal{B}(x\succ_{A}y,z)=\mathcal{B}(x,z\succ_{A}y),\;\forall x,y,z\in A.
    \end{equation}
\end{pro}
\begin{proof}
    Let $\phi:A\rightarrow A^{*}$ be a linear map and $\mathcal{B}$ be a bilinear form on $A$ given by $\mathcal{B}^{\natural}=\phi$.
    Then $\phi$ is bijective if and only if $\mathcal{B}$ is nondegenerate.
    Moreover, the following equations hold
        \begin{eqnarray*}
            &&\phi\big( L_{\succ_{A}}(x)y\big) =( L^{*}_{\circ_{A}}+ R^{*}_{\circ_{A}})(x)\phi(y) =( L^{*}_{\succ_{A}}+ R^{*}_{\succ_{A}})(x)\phi(y),\\
&&\phi\big( R_{\succ_{A}}(y)x\big)
            =- R^{*}_{\succ_{A}}(y)\phi(x),\;
-\phi\big( L_{\prec_{A}}(x)y\big)
            =\phi\big( L_{\prec_{A}}(y)x\big)
            =- R^{*}_{\circ_{A}}(y)\phi(x),\;\forall x,y\in A
    \end{eqnarray*}if and only if  \eqref{eq:cor3}, \eqref{eq:cor4} and
    \eqref{eq:cor3.37} are satisfied.
    Hence the conclusion follows.
\end{proof}

\begin{lem}\label{lem:337}
    Let $(A,\succ_{A},\prec_{A})$ be a \sctpla and $(A,\circ_{A})$ be
    the sub-adjacent Leibniz algebra. Suppose that $\mathcal{B}$ is a symmetric bilinear form on $A$ satisfying \eqref{eq:cor4}.
    Then \eqref{eq:li}, \eqref{eq:cor3} and
    \eqref{eq:cor3.37} hold.
\end{lem}
\begin{proof}
    Let $x,y,z\in A$. Then we have
        \begin{eqnarray*}
            \mathcal{B}(x\circ_{A}y,z)+\mathcal{B}(y,x\circ_{A}z)
            =\mathcal{B}(z,x\circ_{A}y)+\mathcal{B}(y,x\circ_{A}z)
            \overset{\eqref{eq:cor4}}{=}
            \mathcal{B}(-z\prec_{A}y-y\prec_{A}z,x)=0,
    \end{eqnarray*}that is, \eqref{eq:li} holds. Moreover, we have
    \begin{eqnarray*}
        &&\mathcal{B}(x\succ_{A} y,z)=\mathcal{B}(x\circ_{A}y-x\prec_{A}y,z)
        =\mathcal{B}(x\circ_{A}y,z)+\mathcal{B}(x,z\circ_{A}y)\\
        &&=\mathcal{B}(x\circ_{A}y,z)+\mathcal{B}(z\circ_{A}y,x)
        \overset{\eqref{eq:li}}{=}-\mathcal{B}(y,x\circ_{A}z+z\circ_{A}x),
    \end{eqnarray*}
    that is, \eqref{eq:cor3} holds. Note that
    $\mathcal{B}(y,x\circ_{A}z+z\circ_{A}x)$ is symmetric in $x$ and
    $z$. Then we have
    \begin{eqnarray*}
        \mathcal{B}(x\succ_{A} y,z)=\mathcal{B}(z\succ_{A} y,x)=\mathcal{B}(x,z\succ_{A}y),
    \end{eqnarray*}
    that is, \eqref{eq:cor3.37} holds.
\end{proof}

Now we are motivated to give the following notion.

\begin{defi}
    A (symmetric) bilinear form $\mathcal{B}$ on a \sctpla
    $(A,\succ_{A}$,
    $\prec_{A})$ is called
    {\bf invariant} if \eqref{eq:cor4} holds. A {\bf quadratic \sctpla}
    $(A,\succ_{A}$, $\prec_{A}$, $\mathcal{B})$ is a \sctpla $(A,\succ_{A},\prec_{A})$ together with a nondegenerate symmetric invariant bilinear form $\mathcal{B}$.
\end{defi}

\begin{pro}\label{pro:330}
    Let $\mathcal{B}$ be a nondegenerate symmetric left-invariant
    bilinear form on a Leibniz algebra $(A,\circ_{A})$. Define
    multiplications $\succ_{A},\prec_{A}:A\otimes A\rightarrow A$
    respectively by \eqref{eq:cor3} and \eqref{eq:cor4}.
    Then $(A,\succ_{A},\prec_{A},\mathcal{B})$ is a quadratic
    \sctpla. Conversely, let $(A,\succ_{A},\prec_{A},$
    $\mathcal{B})$ be
    a quadratic \sctpla. Then $\mathcal B$ is left-invariant on the
    sub-adjacent Leibniz algebra $(A,\circ_A)$. Hence there is a
    one-to-one correspondence between Leibniz algebras with nondegenerate
    symmetric left-invariant bilinear forms and quadratic \sctplas.
\end{pro}

\begin{proof}
It follows from  Corollary \ref{cor:quadratic1} and Lemma    \ref{lem:337}.
\end{proof}

\subsection{From admissible averaging Lie algebras to \sctplas}\

\begin{defi}
Let $(A,[-,-]_{A},P)$ be an averaging Lie algebra. Suppose that
$(\rho,V)$ is a representation of $(A,[-,-]_{A})$.
If there is a linear map $\alpha:V\rightarrow V$ such that
\begin{equation}\label{eq:rep ao}
    \rho(Px)\alpha(v)=\alpha\big(\rho(Px)v\big)=\alpha\big(\rho(x)\alpha(v)\big),\;\forall x\in A, v\in V,
\end{equation}
then we say that $(\rho,\alpha,V)$ is a {\bf representation} of $(A,[-,-]_{A},P)$.
Two representations $(\rho ,\alpha ,V )$ and $(\rho',\alpha',V')$ of $(A,[-,-]_{A},P)$ are called {\bf equivalent} if there exists a linear isomorphism 
$\phi:V\rightarrow V'$ such that
\begin{equation*}
    \phi\big(\rho(x)v\big)=\rho'(x)\phi(v),\; \phi\alpha(v)=\alpha'\phi(v),\;\forall x\in A, v\in V.
\end{equation*}
\end{defi}

\begin{ex}
Let $(A,[-,-]_{A},P)$ be an averaging Lie algebra.
Then  $(\mathrm{ad}_{A},P,A)$ is a representation of $(A,[-,-]_{A},P)$, which is called the {\bf adjoint representation} of $(A,[-,-]_{A},P)$.
\end{ex}

For vector spaces $V$ and $V'$ and linear maps
$\phi:V\rightarrow V$ and $\phi':V'\rightarrow V'$, define
\begin{equation}\label{eq:pro:SD RB Lie2}
\phi+\phi':V\oplus V'\rightarrow V\oplus
	V',\;\;(\phi+\phi')(v+v'):=\phi(v)+\phi'(v'),\;\;\forall v\in
	V,v'\in V'.
\end{equation}

\begin{pro}
Let $(A,[-,-]_{A},P)$ be an averaging Lie algebra. Suppose that $V$ is a vector space, and $\rho:A\rightarrow\mathrm{End} (V),\; \alpha:V\rightarrow V$ are linear maps.
Then $(\rho,\alpha,V)$ is a representation of $(A,[-,-]_{A},P)$ if and only if $(A\ltimes_{\rho}V,P+\alpha)$ is an averaging Lie algebra.
\end{pro}
\begin{proof}
It follows from a straightforward checking.
\end{proof}

Now we introduce the notion of admissible averaging Lie algebras.
\begin{defi}\label{def:aaLg}
Let $(A,[-,-]_{A},P)$ be an averaging Lie algebra and
$Q:A\rightarrow A$ be a linear map. We say that $Q$ is {\bf
admissible  to $(A,[-,-]_{A},P)$} if
 \begin{equation}\label{eq:ao pair}
     [P(x),Q(y)]_{A}=Q\big([P(x),y]_{A}\big)=Q\big([x,Q(y)]_{A}\big),\;\forall x,y\in A.
 \end{equation}
 In this case, we also call $(A,[-,-]_{A},P,Q)$ an {\bf admissible averaging Lie algebra}.
\end{defi}

\begin{cor}
    Let $(A,[-,-]_{A},P)$ be an averaging Lie algebra.
    Then $(\mathrm{ad}^{*}_{A},Q^{*},A^{*})$ is a representation of $(A,[-,-]_{A},P)$ if and only if \eqref{eq:ao pair} holds such that $(A,[-,-]_{A},P,Q)$ is an admissible averaging Lie algebra.
\end{cor}
\begin{proof}
    It follows from a direct computation.
\end{proof}

\begin{pro}\label{pro:2.6}
    Let $(A,[-,-]_{A},P)$ be an averaging Lie algebra and $\mathcal{B}$ be a nondegenerate invariant bilinear form on the Lie algebra $(A,[-,-]_{A})$.
    Then $(A,[-,-]_{A},P,\widehat{P})$ is an admissible averaging Lie algebra, where $\widehat{P}$ is the adjoint map of $P$ associated to $\mathcal{B}$.
    Moreover, $(\mathrm{ad}^{*}_{A},\widehat{P}^{*},A^{*})$ and $(\mathrm{ad}_{A},P,A)$ are equivalent as representations of $(A,[-,-]_{A},P)$.

    Conversely, let $(A,[-,-]_{A},P,Q)$ be an admissible averaging Lie algebra.
    If the resulting representation $(\mathrm{ad}^{*}_{A},Q^{*},A^{*})$ of $(A,[-,-]_{A},P)$ is equivalent to $(\mathrm{ad}_{A},P,A)$, then there exists a nondegenerate invariant bilinear form $\mathcal{B}$ on $(A,[-,-]_{A})$ such that $Q=\widehat{P}$.
\end{pro}

\begin{proof}
    The proof is similar to that of \cite[Proposition 3.9]{Bai2021}.
\end{proof}

There is the following relationship between admissible averaging
Lie algebras and \sctplas.

\begin{pro}\label{pro:com asso and SDPP}
    Let $(A,[-,-]_{A},P,Q)$ be an admissible averaging Lie algebra.
   Define two  multiplications $\succ_{A},$
    $\prec_{A}:A\otimes A\rightarrow A$ on $A$  by
    \begin{equation}\label{eq:com asso and SDPP}
        x\succ_{A}y=[P(x),y]_{A}-Q([x,y]_{A}),\;
        x\prec_{A}y=Q([x,y]_{A}),\;\forall x,y\in A.
    \end{equation}
    Then
    $(A,\succ_{A},\prec_{A})$ is a \sctpla.
\end{pro}

\begin{proof}
    By  Proposition~\ref{ex:Lie aver}, $(A,\circ_{A}=\succ_{A}+\prec_{A})$ satisfies \eqref{eq:Leibniz from aver op} and hence is a Leibniz algebra.
    By the assumption,
    $(A\ltimes_{\mathrm{ad}^*_{A}}A^*,P+Q^*)$ is an averaging
    Lie algebra. Thus there is a Leibniz algebra
    $(A\oplus A^*,\circ_{d})$ with $\circ_d$ defined by
\begin{eqnarray*}
(x+a^{*})\circ_{d}(y+b^{*})&=&[(P+Q^{*})(x+a^{*}),y+b^{*}]_{d} \\
&=&[P(x)+Q^{*}(a^{*}),y+b^{*}]_{d} \\
&=&[P(x),y]_{A}+\mathrm{ad}_{A}^{*}\big(P(x)\big)
b^{*}-\mathrm{ad}_{A}^{*}(y)Q^{*}(a^{*}) \\
&\overset{\eqref{eq:Leibniz from aver op},\eqref{eq:com asso and SDPP}}{=}&x\circ_{A}y+
 L_{\circ_{A}}^{*}(x)b^{*}- L_{\prec_{A}}^{*}(y)a^{*},
\;\;\forall x,y\in A, a^*,b^*\in A^*.
\end{eqnarray*}
    That is, $( L^{*}_{\circ_{A}},- L^{*}_{\prec_{A}},A^{*})$ is a representation of $(A,\circ_{A})$. Hence  $(A,\succ_{A},\prec_{A})$ is a \sctpla which is compatible with $(A,\circ_{A})$.
\end{proof}

\begin{ex}\label{ex:ex2.6}
Let $A$ be a vector space with a multiplication $\circ_{A}:A\otimes A\rightarrow A$.
Let $(\mathrm{End}_{\mathbb
K}(A)\oplus A,[-,-]_{d})$ be the Lie algebra with $[-,-]_{d}$  given by
\begin{eqnarray}
[f+x,g+y]_{d}=[f,g]+f(y)-g(x),\;\forall f,g\in \mathrm{End}_{\mathbb
K}(A),x,y\in A,
\end{eqnarray}
where $[f,g]:=fg-gf$ is the canonical Lie bracket on $\mathrm{End}_{\mathbb
K}(A)$. Define linear maps
 $P,Q: \mathrm{End}_{\mathbb
K}(A)\oplus A\rightarrow \mathrm{End}_{\mathbb
K}(A)\oplus A $ by
 \begin{eqnarray}
    P(f+x)= L_{\circ_{A}}(x),\;
    Q(f+x)=- R_{\circ_{A}}(x),\;\forall f\in \mathrm{End}_{\mathbb K}(A), x\in A.
 \end{eqnarray}
By a straightforward computation, $(\mathrm{End}_{\mathbb
K}(A)\oplus A,[-,-]_{d},P,Q)$
is an admissible averaging Lie algebra if and only if $(A,\circ_{A})$ is a Leibniz algebra.
In this case, there is a
\sctpla $(\mathrm{End}_{\mathbb
    K}(A)\oplus A,\succ_{d},\prec_{d})$ induced from $(\mathrm{End}_{\mathbb K}(A)\oplus A,[-,-]_{d},P,Q)$, where the multiplications $\succ_{d},\prec_{d}$ are explicitly given by
\begin{eqnarray*}
    &&(f+x)\prec_{d}(g+y)
    = R_{\circ_{A}}\big(g(x)\big)- R_{\circ_{A}}\big(f(y)\big),\\
    &&(f+x)\succ_{d}(g+y)
    =[ L_{\circ_{A}}(x),g]+x\circ_{A}y
    + R_{\circ_{A}}\big(f(y)\big)- R_{\circ_{A}}\big(g(x)\big),\;\forall f,g\in \mathrm{End}_{\mathbb
        K}(A),x,y\in A.
\end{eqnarray*}
Thus we have the sub-adjacent Leibniz algebra
$(\mathrm{End}_{\mathbb K}(A)\oplus A,\circ_{d})$
of $(\mathrm{End}_{\mathbb
    K}(A)\oplus A,\succ_{d},\prec_{d})$, where $\circ_d$ is given by
\begin{eqnarray*}
    (f+x)\circ_{d}(g+y):=[ L_{\circ_{A}}(x),g]+x\circ_{A}y,\;\forall f,g\in \mathrm{End}_{\mathbb
        K}(A),x,y\in A.
\end{eqnarray*}
\end{ex}

\section{\Sctplbs and their realizations from averaging Lie bialgebras}\label{sec:5}\

We introduce the notion of a \sctplb, which is equivalently characterized as a
  Manin
triple  of Leibniz algebras associated to the nondegenerate
symmetric left-invariant bilinear form, as well as a  Manin triple
of \sctplas. Then we introduce the notion of an averaging Lie
bialgebra, which is equivalently characterized as a Manin triple
of averaging Lie algebras. Finally, we demonstrate that a \sctplb
can be obtained from an averaging Lie bialgebra, and equivalently,
a Manin triple of \sctplas can be obtained from a Manin triple of
averaging Lie algebras. 

\subsection{\Sctplbs and their characterizations in terms of Manin triples}\label{sec3.1}\

Recall that a {\bf Leibniz coalgebra} \cite{ST} is a vector
space $A$ with a co-multiplication $\eta:A\rightarrow A\otimes A$
such that the following equation holds:
\begin{eqnarray}\label{eq:leco2}
    (\eta\otimes\mathrm{id})\eta(x)+(\tau\otimes\mathrm{id})
    (\mathrm{id}\otimes\eta)\eta(x)-(\mathrm{id}\otimes\eta)\eta(x)=0
    ,\;\forall x\in A.
\end{eqnarray}

Now we introduce the notion of a \sctplc.

\begin{defi}
    Let $A$ be a vector space with co-multiplications $\vartheta,\theta:A\rightarrow A\otimes A$, and $\eta=\vartheta+\theta$.
    If $(A,\eta)$ is a Leibniz coalgebra and the following equations hold:
    \begin{eqnarray}
        \theta(x)&=&-\tau\theta(x),\label{eq:co1} \\
        (\mathrm{id}\otimes\theta)\eta(x)&=&(\eta\otimes\mathrm{id})  \theta(x)+(\tau\otimes\mathrm{id})(\mathrm{id}\otimes\eta)\theta(x)
        ,\label{eq:co4}\\
        (\mathrm{id}\otimes\theta)\vartheta(x)&=&0
        ,\;\forall x\in A,\label{eq:co5}
    \end{eqnarray}
    then we call $(A,\vartheta,\theta)$ a {\bf \sctplc}.
\end{defi}

\begin{pro}\label{lem:co}
    Let $A$ be a vector space and $\vartheta,\theta :A\rightarrow A\otimes A$ be co-multiplications.
    Let $\succ_{A^{*}},\prec_{A^{*}}:A^{*}\otimes A^{*}\rightarrow A^{*}$ be the linear duals of $\vartheta$ and $\theta$ respectively, that is, the following equations hold:
    \begin{equation}
        \langle a^{*}\succ_{A^{*}}b^{*},x\rangle=\langle a^{*}\otimes b^{*},\vartheta(x)\rangle,\;\langle a^{*}\prec_{A^{*}}b^{*},x\rangle=\langle a^{*}\otimes b^{*},\theta(x)\rangle,\;\forall x\in A, a^{*},b^{*}\in A^{*}.
    \end{equation}
    Then $(A^{*},\succ_{A^{*}},\prec_{A^{*}})$ is a \sctpla if and only
    if $(A,\vartheta,\theta )$ is a \sctplc.
\end{pro}

\begin{proof}
It follows directly from Corollary \ref{cor:SDPP}.
\end{proof}

Now we give the notion of a \sctplb.

\begin{defi}
    Let $(A,\succ_{A},\prec_{A})$ be a \sctpla. Suppose that there is a
    \sctplc $(A,\vartheta,\theta)$ and the following equations hold:
{\small
    \begin{align}
&\eta(x\prec_{A}y)=
        (\mathrm{id}\otimes\mathrm{id}-\tau)\big(\mathrm{id}\otimes L_{\prec_{A}}(x)\big)\eta(y)
        +(\mathrm{id}\otimes\mathrm{id}-\tau)\big(\mathrm{id}\otimes R_{\prec_{A}}(y)\big)\eta(x),\label{eq:bialg1}\\
        & \big( L_{\circ_{A}}(x)\otimes\mathrm{id}\big)\eta(y)=
        \big( L_{\prec_{A}}(x)  \otimes\mathrm{id}\big)\eta(y)
        +\big(\mathrm{id}\otimes R_{\circ_{A}}(y)\big)\eta(x)
        -\big(\mathrm{id}\otimes R_{\circ_{A}}(y)\big)\theta(x),\label{eq:bialg2}\\
        & \theta (x\circ_{A}y)=\big(\mathrm{id}\otimes L_{\circ_{A}}(x)
        + L_{\circ_{A}}(x)\otimes\mathrm{id}\big)\theta(y)
        -\big(\mathrm{id}\otimes L_{\circ_{A}}(y)
        + L_{\circ_{A}}(y)\otimes\mathrm{id}\big)\theta(x),\label{eq:bialg3}\\
        &\eta(x\circ_{A}y)=\big(\mathrm{id}\otimes R_{\circ_{A}}(y)\big)\eta(x)
        +\big(\mathrm{id}\otimes L_{\circ_{A}}(x)
        + L_{\prec_{A}}(x)\otimes\mathrm{id}\big)\eta(y)
        -\big( L_{\prec_{A}}(y)\otimes\mathrm{id}\big)\theta(x),\label{eq:bialg4}
    \end{align}}
for all $x,y\in A$. Then we say that  $(A,\succ_{A},\prec_{A},\vartheta,\theta)$ is a \textbf{\sctplb}.
\end{defi}

\begin{lem}\label{lem:mp}\cite{AM}
    Let $(A,\circ_{A})$ and $(B,\circ_{B})$ be two Leibniz algebras and
    $l_{A},r_{A}:A\rightarrow \mathrm{End}_{\mathbb K}(B),$
$l_{B},r_{B}:B\rightarrow \mathrm{End}_{\mathbb K}(A)$ be linear maps.
    Then there is a Leibniz algebra structure on $A\oplus B$ given by
    \begin{small}
        \begin{equation}        (x+a)\circ(y+b):=x\circ_{A}y+l_{B}(a)y+r_{B}(b)x+a\circ_{B}b+l_{A}(x)b+r_{A}(y)a,\;\forall x,y\in A, a,b\in B
        \end{equation}
    \end{small}if and only if $(l_{A},r_{A},B)$ and $(l_{B},r_{B},A)$ are representations of $(A,\circ_{A})$ and $(B,\circ_{B})$ respectively, and the following equations
    hold
    \begin{eqnarray}
        r_{A}(x)(a\circ_{B}b)-a\circ_{B}r_{A}(x)b+b\circ_{B}r_{A}(x)a-r_{A}(l_{B}(b)x)a+r_{A}(l_{B}(a)x)b=0,&&\label{eq:mp1}\\
        l_{A}(x)(a\circ_{B}b)-l_{A}(x)a\circ_{B}b-a\circ_{B}l_{A}(x)b-l_{A}(r_{B}(a)x)b-r_{A}(r_{B}(b)x)a=0,&&\label{eq:mp2}\\
        l_{A}(x)a\circ_{B}b+l_{A}(r_{B}(a)x)b+r_{A}(x)a\circ_{B}b+l_{A}(l_{B}(a)x)b=0,&&\label{eq:mp3}\\
        r_{B}(a)(x\circ_{A}y)-x\circ_{A}r_{B}(a)y+y\circ_{A}r_{B}(a)x-r_{B}(l_{A}(y)a)x+r_{B}(l_{A}(x)a)y=0,&&\label{eq:mp4}\\
        l_{B}(a)(x\circ_{A}y)-l_{B}(a)x\circ_{A}y-x\circ_{A}l_{B}(a)y-l_{B}(r_{A}(x)a)y-r_{B}(r_{A}(y)a)x=0,&&\label{eq:mp5}\\
        l_{B}(a)x\circ_{A}y+l_{B}(r_{A}(x)a)y+r_{B}(a)x\circ_{A}y+l_{B}(l_{A}(x)a)y=0,&&\label{eq:mp6}
    \end{eqnarray}
for all $x, y\in A, a, b\in B$.
\end{lem}

\begin{pro}\label{pro:2-2}
    Let $(A,\succ_{A},\prec_{A})$ and
    $(A^{*},\succ_{A^{*}},\prec_{A^{*}})$ be two \sctplas, and linear
    maps $\vartheta,\theta:A\rightarrow A\otimes A$ be the linear
    duals of $\succ_{A^{*}}$ and $\prec_{A^{*}}$ respectively. Then
    $(A,\succ_{A}$, $\prec_{A}$, $\vartheta$, $\theta)$ is a \sctplb if and only if there is a Leibniz algebra $(A\oplus A^{*},\circ_{d})$ with $\circ_{d}$
    defined by \begin{equation}\label{eq:A ds}
        (x+a^{*})\circ_{d}(y+b^{*}):
        =x\circ_{A}y+ L^{*}_{\circ_{A^{*}}}(a^{*})y
        - L^{*}_{\prec_{A^{*}}}(b^{*})x
        +a^{*}\circ_{A^{*}}b^{*}+
         L^{*}_{\circ_{A}}(x)b^{*}
        - L^{*}_{\prec_{A}}(y)a^{*},
    \end{equation}
    for all $x,y\in A, a^{*},b^{*}\in A^{*}$.
\end{pro}

\begin{proof}
    The conclusion follows from Lemma \ref{lem:mp} by taking
    $ l_{A}=
    L^{*}_{\circ_{A}},r_{A}=- L^{*}_{\prec_{A}}, l_{B}= L^{*}_{\circ_{A^{*}}},
    r_{B}=- L^{*}_{\prec_{A^{*}}}$. Note that in this case, we
    have
\begin{eqnarray*}
 \eqref{eq:bialg1} \Longleftrightarrow \eqref{eq:mp1},\;   \eqref{eq:bialg2}\Longleftrightarrow    \eqref{eq:mp3}\Longleftrightarrow\eqref{eq:mp6},\;
       \eqref{eq:bialg3} \Longleftrightarrow \eqref{eq:mp4},\;\eqref{eq:bialg4}\Longleftrightarrow\eqref{eq:mp2}
        \Longleftrightarrow\eqref{eq:mp5}.
    \end{eqnarray*}
Hence the conclusion follows.
\end{proof}

Now we study the equivalent characterization of \sctplbs in terms of the Manin triples.

\begin{defi}
\begin{enumerate}
\item Let $(A,\circ_{A})$ and $(A^{*},\circ_{A^{*}})$ be Leibniz
algebras. If there is a Leibniz algebra structure $(A\oplus
A^{*},\circ_{d})$ on $A\oplus A^{*}$ containing $(A,\circ_{A})$
and $( A^{*},\circ_{A^{*}})$ as Leibniz subalgebras, and the
natural nondegenerate symmetric bilinear form $\mathcal{B}_{d}$
defined by
 \begin{equation}\label{eq:bfds}
        \mathcal{B}_{d}(x+a^{*},y+b^{*})=\langle x,b^{*}\rangle+\langle a^{*},y\rangle,\;\forall x,y\in A, a^{*},b^{*}\in A^{*}
    \end{equation}
is left-invariant on $(A\oplus
A^{*},\circ_{d})$, then we say that
$\big(  (  A\oplus
A^{*},\circ_{d},\mathcal{B}_{d}),(A,\circ_{A}),
(A^{*},\circ_{A^{*}})\big)
$ is a {\bf Manin triple of Leibniz algebras associated to the nondegenerate
symmetric left-invariant bilinear form $\mathcal{B}_{d}$}.
\item Let $(A,\succ_{A},\prec_{A})$ and
$(A^{*},\succ_{A^{*}},\prec_{A^{*}})$ be \sctplas. If there is a
quadratic \sctpla structure $(A\oplus
A^{*},\succ_{d},\prec_{d},\mathcal{B}_{d})$ on $A\oplus A^{*}$
which contains $(A,\succ_{A},$
$\prec_{A})$ and $(
A^{*},\succ_{A^{*}},\prec_{A^{*}})$ as \sctplasubs, then we say that
$\big(  (  A\oplus
A^{*},\succ_{d},\prec_{d},$
$\mathcal{B}_{d}),
(A,\succ_{A},\prec_{A}),(A^{*},\succ_{A^{*}},\prec_{A^{*}})\big)
$ is a \textbf{Manin triple of \sctplas} (associated to the nondegenerate symmetric invariant bilinear form $\mathcal{B}_{d}$). \end{enumerate}
\label{d:manintypea}
\end{defi}

\begin{pro}\label{pro:equ}
There is a one-to-one correspondence between
Manin triples of Leibniz algebras associated to the nondegenerate
symmetric left-invariant bilinear forms and Manin triples of
\sctplas.
\end{pro}

\begin{proof}
    Let $\big(  (  A\oplus
    A^{*},\circ_{d},\mathcal{B}_{d}),(A,\circ_{A}),
    (A^{*},\circ_{A^{*}})\big)
    $ be a Manin triple of Leibniz algebras associated to the nondegenerate
    symmetric left-invariant bilinear form. Then  by Proposition \ref{pro:330}, there is a quadratic \sctpla  $(A\oplus A^{*},\succ_{d},\prec_{d},\mathcal{B})$ with multiplications $\succ_{d},\prec_{d}$ defined by
    \begin{eqnarray}
        && \mathcal{B}_{d}\big( (x+a^{*})\prec_{d}(y+b^{*}) ,z+c^{*}\big)=-\mathcal{B}_{d}\big( x+a^{*}, (z+c^{*})\circ_{d}(y+b^{*})\big),\label{eq:3603}\\
        &&(x+a^{*})\succ_{d}(y+b^{*})=
        (x+a^{*})\circ_{d}(y+b^{*})-(x+a^{*})\prec_{d}(y+b^{*}),\label{eq:3605}
    \end{eqnarray}
for all $ x,y,z\in A, a^{*},b^{*},c^{*}\in A^{*}.    $
   In particular, we have
\begin{eqnarray*}
    \mathcal{B}_{d}(x\prec_{d}y,z)
    =-\mathcal{B}_{d}(x,z\circ_{d}y)=0,\;\;\forall
    x,y,z\in A.
\end{eqnarray*}
Hence for all $x,y\in A$, 
we have $x\prec_{d}y\in A$, and $x\succ_{d}y=x\circ_{d}y-x\prec_{d}y
=x\circ_{A}y-x\prec_{d}y\in A$.
Thus $(A,\succ_{A},\prec_{A})$
$=
(A,\succ_{d}|_{A\otimes A},\prec_{d}|_{A\otimes A})$ is a subalgebra of $(A\oplus A^{*},\succ_{d},\prec_{d})$.
Similarly, $(A^{*},\succ_{A^{*}},\prec_{A^{*}})=
(A^{*},\succ_{d}|_{A^{*}\otimes A^{*}},\prec_{d}|_{A^{*}\otimes A^{*}})$ is also a subalgebra of $(A\oplus A^{*},\succ_{d},\prec_{d})$.
Hence $\big(  (  A\oplus
A^{*},\succ_{d},\prec_{d},\mathcal{B}_{d}),
(A,\succ_{A},$
$\prec_{A}),
(A^{*},\succ_{A^{*}},\prec_{A^{*}})\big)
$ is a  Manin triple of \sctplas.

Conversely, suppose that $\big(  (  A\oplus
A^{*},\succ_{d},\prec_{d},\mathcal{B}_{d}),
(A,\succ_{A},\prec_{A}), (A^{*},\succ_{A^{*}},\prec_{A^{*}})\big)
$ is a  Manin triple of \sctplas. Then $\big(  (  A\oplus
A^{*},\circ_{d},\mathcal{B}_{d}),
(A,\circ_{A}),(A^{*},\circ_{A^{*}})\big) $ is obviously  a Manin
triple of Leibniz algebras associated to the nondegenerate
symmetric left-invariant bilinear form, where $(A\oplus
A^{*},\circ_{d}),(A,\circ_{A})$ and $(A^{*},\circ_{A^{*}})$ are
the sub-adjacent Leibniz algebras of $(A\oplus A^{*},\succ_{d}$,
$\prec_{d})$, $(A,\succ_{A}$, $\prec_{A})$ and
$(A^{*},\succ_{A^{*}},\prec_{A^{*}})$ respectively.
\end{proof}

\begin{thm}\label{thm:Manin triple}
Let $(A,\succ_{A},\prec_{A})$ and
$(A^{*},\succ_{A^{*}},\prec_{A^{*}})$ be \sctplas and their sub-adjacent Leibniz algebras be $(A,\circ_{A})$ and
$(A^{*},\circ_{A^{*}})$ respectively. Then the
following conditions are equivalent:
\begin{enumerate}
    \item\label{E5}
    $(A,\succ_{A},\prec_{A},\vartheta,\theta)$ is a \sctplb, where $\vartheta,\theta:A\rightarrow A\otimes A$ are the linear
    duals of $\succ_{A^{*}}$ and $\prec_{A^{*}}$ respectively.
\item\label{E4} There is a Manin triple $\big(  (  A\oplus
A^{*},\circ_{d},\mathcal{B}_{d}),(A,\circ_{A}),
(A^{*},\circ_{A^{*}})\big)
$ of Leibniz algebras associated to the nondegenerate symmetric
left-invariant bilinear form such that the compatible \sctpla
$(A\oplus A^{*},\succ_{d},\prec_{d})$ induced
from $\mathcal{B}_{d}$ contains
$(A,\succ_{A},\prec_{A})$ and
$(A^{*},\succ_{A^{*}},\prec_{A^{*}})$ as
\sctplasubs.
\item\label{E1} There is a Manin triple
$\big((A\oplus
A^{*},\succ_{d},\prec_{d},\mathcal{B}_{d}),
(A,\succ_{A},\prec_{A}),
(A^{*},\succ_{A^{*}},\prec_{A^{*}})\big)$ of \sctplas.
\item\label{E2} There is
a Leibniz algebra $(A\oplus A^{*},\circ_{d})$ with $ \circ_{d} $
defined by \eqref{eq:A ds}.
\item\label{E3} There is
a \sctpla  $(A\oplus A^{*},\succ_{d},\prec_{d})$
with $\succ_{d},\prec_{d}$ respectively defined by
\begin{eqnarray}
 (x+a^{*})\succ_{d}(y+b^{*})&=&x\succ_{A}y
+( L^{*}_{\succ_{A^{*}}}+ R^{*}_{\succ_{A^{*}}})(a^{*})y
- R^{*}_{\succ_{A^{*}}}(b^{*})x\nonumber\\
&&+a^{*}\succ_{A^{*}}b^{*}+( L^{*}_{\succ_{A}}
+ R^{*}_{\succ_{A}})(x)b^{*}- R^{*}_{\succ_{A}}(y)a^{*},\label{eq:A ds1}\\
(x+a^{*})\prec_{d}(y+b^{*})&=&x\prec_{A}y
- R^{*}_{\circ_{A^{*}}}(a^{*})y
+ R^{*}_{\circ_{A^{*}}}(b^{*})x\nonumber\\
&&+a^{*}\prec_{A^{*}}b^{*}
- R^{*}_{\circ_{A}}(x)b^{*}
+ R^{*}_{\circ_{A}}(y)a^{*},\;\forall x,y\in A, a^{*},b^{*}\in A^{*}.\label{eq:A ds2}
\end{eqnarray}
\end{enumerate}
\end{thm}

\begin{proof}
    (\ref{E5})$\Longleftrightarrow$(\ref{E2}) It follows from Proposition
    \ref{pro:2-2}.

(\ref{E1})$\Longleftrightarrow$(\ref{E4}) It follows from
Proposition \ref{pro:equ}.

(\ref{E2})$\Longrightarrow$(\ref{E1}) Suppose that there is a
Leibniz algebra $(A\oplus A^{*},\circ_{d})$ with $\circ_{d}$
defined by \eqref{eq:A ds}. Then it is straightforward to check
that $\mathcal{B}_{d}$ is left-invariant on $(A\oplus
A^{*},\circ_{d})$. By Proposition \ref{pro:330}, there is a
quadratic \sctpla $(A\oplus
A^{*},\succ_{d},\prec_{d},\mathcal{B}_{d})$ with
$\succ_{d},\prec_{d}$ defined by \eqref{eq:cor3} and
\eqref{eq:cor4} respectively. It is also straightforward to show
that
$$x\succ_{d}y,\; x\prec_{d}y\in A,\;a^{*}\succ_{d}b^{*},\;a^{*}\prec_{d}b^{*}\in A^{*},\;\forall x,y\in A,
a^{*},b^{*}\in A^{*}.$$
Hence $(A,\succ_{A}=\succ_{d}|_{A\otimes A},\prec_{A}=\prec_{d}|_{A\otimes A})$ and
$(A^{*},\succ_{A^{*}}=\succ_{d}|_{A^{*}\otimes A^{*}},\prec_{A^{*}}=\prec_{d}|_{A^{*}\otimes A^{*}})$ are \sctplasubs of
$(A\oplus A^{*},\succ_{d},\prec_{d})$. Therefore $\big((A\oplus
A^{*},\succ_{d}$, $\prec_{d},\mathcal{B}_{d}),
(A,\succ_{A},\prec_{A}), (A^{*},\succ_{A^{*}},\prec_{A^{*}})\big)
$ is a Manin triple of \sctplas.

(\ref{E1})$\Longrightarrow$(\ref{E3})  Let $x,y\in A,
a^{*},b^{*}\in A^{*}$. By the assumption, we have
\begin{eqnarray*}
    \mathcal{B}_{d}(x\succ_{d} b^{*},y)&\overset{\eqref{eq:cor3}}{=}&-\mathcal{B}
    _{d}(b^{*},x\circ_{A}y+y\circ_{A}x)=
    \mathcal{B}_{d}\big(( L
    ^{*}_{\succ_{A}}+ R
    ^{*}_{\succ_{A}})(x)b^{*},y\big),\\
    \mathcal{B}_{d}(x\succ_{d} b^{*},a^{*})&\overset{\eqref{eq:cor3.37}}{=}&
    \mathcal{B}_{d}(x,a^{*}\succ_{A^{*}}b^{*})=\langle x,
    a^{*}\succ_{A^{*}}b^{*}\rangle
    =-\mathcal{B}_{d}\big( R
    ^{*}_{\succ_{A^{*}}}(b^{*})x,a^{*}\big).
\end{eqnarray*}
 Thus we have
\begin{eqnarray*}
x\succ_{d} b^{*}=( L ^{*}_{\succ_{A}}+ R
^{*}_{\succ_{A}})(x)b^{*}- R ^{*}_{\succ_{A^{*}}}(b^{*})x.
\end{eqnarray*}
Similarly we have
\begin{eqnarray*}
	a^{*}\succ_{d} y=( L
	^{*}_{\succ_{A^{*}}}+ R ^{*}_{\succ_{A^{*}}})(a^{*})y- R
	^{*}_{\succ_{A}}(y)a^{*},\; x\prec_{d} b^{*}=
	R^{*}_{\circ_{A^{*}}} (b^{*})x- R^{*}_{\circ_{A}} (x)b^{*}.
\end{eqnarray*}
Hence \eqref{eq:A ds1} and  \eqref{eq:A ds2} hold.

(\ref{E3})$\Longrightarrow$(\ref{E2}) It follows from a
straightforward verification.
\end{proof}

\subsection{From averaging Lie bialgebras to \sctplbs}\

Recall that
a {\bf Lie coalgebra} \cite{Cha} is a pair $(A,\delta)$, where $A$ is a vector space and $\delta:A\rightarrow A\otimes A$ is a co-multiplication such that the following equations hold:
    \begin{equation}\label{eq:coLie} \delta=-\tau\delta,\;(\mathrm{id}+\sigma+\sigma^{2})(\mathrm{id}\otimes \delta)\delta=0,
    \end{equation}
    where 
 $\sigma(x\otimes y\otimes z)=z\otimes x\otimes y$ for all $x,y,z\in A$.
    A {\bf Lie bialgebra} \cite{Cha} is a triple $(A,[-,-]_{A},\delta)$ such that $(A,[-,-]_{A})$ is a Lie algebra, $(A,\delta)$ is a Lie coalgebra and the following equation holds:
        \begin{equation}\label{eq:bib}          \delta([x,y]_{A})=\big(\mathrm{ad}_{A}(x)\otimes\mathrm{id}
            +\mathrm{id}\otimes\mathrm{ad}_{A}(x)\big)\delta(y)
            -\big(\mathrm{ad}_{A}(y)\otimes\mathrm{id}      +\mathrm{id}\otimes\mathrm{ad}_{A}(y)\big)\delta(x),\;\forall x,y\in A.
        \end{equation}

Now we introduce the notions of an averaging Lie coalgebra and an averaging Lie bialgebra.

\begin{defi}
    An {\bf averaging Lie coalgebra} is a triple $(A,\delta,Q)$ such that $(A,\delta)$ is a Lie coalgebra and $Q:A\rightarrow A$ is a linear map satisfying the following equation:
    \begin{eqnarray}\label{eq:aoco1}
    (Q\otimes Q)\delta(x)=(Q\otimes\mathrm{id})\delta\big(Q(x)\big),\;\forall x\in A.
    \end{eqnarray}
    An {\bf admissible averaging Lie coalgebra} is a quadruple $(A,\delta,Q,P)$
    such that $(A,\delta,Q)$ is an averaging Lie coalgebra and $P:A\rightarrow A$ is a linear map satisfying the following equation:
     \begin{eqnarray} \label{eq:aoco2}
        (Q\otimes P)\delta(x)=(Q\otimes\mathrm{id})\delta\big(P(x)\big)=(\mathrm{id}\otimes P)\delta\big(P(x)\big),\;\forall x\in A.
    \end{eqnarray}
\end{defi}

\begin{pro}\label{pro:admi ave coalg}
    Let $A$ be a vector space, $\delta:A\rightarrow A\otimes A$ be a co-multiplication and  $P,Q:A\rightarrow A$ be linear maps.
    Let $[-,-]_{A^{*}}:A^{*}\otimes A^{*}\rightarrow A^{*}$ be the linear dual of
    $\delta$.
Then $(A^{*},[-,-]_{A^{*}},Q^* )$ is an  averaging Lie algebra if
and only if $(A,\delta,Q )$ is an averaging Lie coalgebra.
Moreover, $(A^{*},[-,-]_{A^{*}},Q^*,P^*)$ is an admissible
averaging Lie algebra if and only if $(A,\delta,Q,P)$ is an
admissible averaging Lie coalgebra.
\end{pro}

\begin{proof}
It follows from a straightforward checking.
\end{proof}

\begin{defi}
    An {\bf averaging Lie bialgebra} is a vector space $A$ together with linear maps
    \begin{equation*}
        [-,-]_{A}:A\otimes A\rightarrow A,\; \delta:A\rightarrow A\otimes A, \; P,Q:A\rightarrow A
    \end{equation*}
    such that the following conditions are satisfied:
    \begin{enumerate}
        \item the triple $(A,[-,-]_{A},\delta)$ is a Lie bialgebra.
        \item \eqref{eq:Ao} and \eqref{eq:ao pair} hold such that $(A,[-,-]_{A},P,Q)$ is an admissible averaging Lie algebra.
            \item \eqref{eq:aoco1} and \eqref{eq:aoco2} hold such that $(A,\delta,Q,P)$ is an admissible averaging Lie coalgebra.
    \end{enumerate}
    We denote it by $(A,[-,-]_{A},\delta,P,Q)$.
\end{defi}

\begin{pro}\label{pro:5.9}
    Let $(A,[-,-]_{A},\delta,P,Q)$ be an averaging Lie bialgebra.
    Then there is a \sctplb $(A,\succ_{A},\prec_{A},\vartheta,\theta)$, where $\succ_{A},\prec_{A}$ are defined by \eqref{eq:com asso and SDPP} and $\vartheta,\theta$ are defined by
    \begin{equation}\label{eq:co ao}
        \vartheta(x):=(Q\otimes\mathrm{id})\delta(x)-\delta(Px),\; \theta(x):=\delta(Px),\;\forall x\in A.
    \end{equation}
\end{pro}
\begin{proof}
By Propositions \ref{pro:com asso and SDPP} and \ref{pro:admi ave coalg}, $(A,\succ_{A},\prec_{A})$ with $\succ_{A},\prec_{A}$ defined by \eqref{eq:com asso and SDPP} and
    $(A^{*},\succ_{A^{*}}$,
$\prec_{A^{*}})$ with $ \succ_{A^{*}},\prec_{A^{*}} $ defined by
\begin{eqnarray}\label{eq:mp re2}
a^{*}\succ_{A^{*}}b^{*}=[Q^{*}(a^{*}),b^{*}]_{A^{*}}
-P^{*}([a^{*},b^{*}]_{A^{*}}),\; a^{*}\prec_{A^{*}}b^{*}=P^{*}([a^{*},b^{*}]_{A^{*}}),\;\forall a^{*},b^{*}\in A^{*}
\end{eqnarray}
respectively are \sctplas.
   Note that the
    linear duals $\vartheta,\theta:A\rightarrow A\otimes A$ of
    $\succ_{A^{*}}$ and $\prec_{A^{*}}$ satisfy the following equations:
        \begin{eqnarray*}
        \langle \theta(x), a^{*}\otimes b^{*}\rangle&=&\langle x, a^{*}\prec_{A^{*}}b^{*}\rangle\overset{\eqref{eq:mp re2}}{=}\langle x, P^{*}([a^{*},b^{*}]_{A^{*}}) \rangle=\langle \delta(Px),a^{*}\otimes b^{*}\rangle,\\
        \langle \vartheta(x), a^{*}\otimes b^{*}\rangle&=&\langle x, a^{*}\succ_{A^{*}}b^{*}\rangle\overset{\eqref{eq:mp re2}}{=}\langle x, [Q^{*}(a^{*}),b^{*}]_{A^{*}}-P^{*}([a^{*},b^{*}]_{A^{*}})\rangle\\
        &=&\langle (Q\otimes\mathrm{id})\delta(x)-\delta(Px),a^{*}\otimes b^{*}\rangle,\;\;\forall x\in A, a^{*},b^{*}\in A^{*}.
    \end{eqnarray*}
Therefore,   $\vartheta,\theta$ are exactly defined by \eqref{eq:co ao}.

By Proposition \ref{lem:co}, $(A,\vartheta,\theta)$ is a \sctplc.
    By \eqref{eq:Leibniz from aver op}, \eqref{eq:com asso and SDPP} and \eqref{eq:mp re2}, we have
 \begin{eqnarray*}
     L_{\succ_{A}}(y)
    =\mathrm{ad}_{A}\big(P(y)\big)-Q\mathrm{ad}_{A}(y),\;
    R_{\succ_{A}}(y) =Q\mathrm{ad}_{A}(y)-\mathrm{ad}_{A}(y)P,\;
    L_{\prec_{A}}(y)
    =- R_{\prec_{A}}(y)=Q\mathrm{ad}_{A}(y),
\end{eqnarray*}
for all $y\in A$.
    Therefore we have
    \begin{align*}
        &\eta(x\prec_{A}y)=(Q\otimes\mathrm{id})\delta(x\prec_{A}y)
        \overset{\eqref{eq:com asso and SDPP}}{=}(Q\otimes\mathrm{id})\delta\big(Q([x,y]_{A})\big)
        \overset{\eqref{eq:aoco1}}{=}(Q\otimes Q)\delta([x,y]_{A})\\
        &\overset{\eqref{eq:bib}}{=}(Q\otimes Q)\big(\mathrm{ad}_{A}(x)\otimes\mathrm{id}
        +\mathrm{id}\otimes\mathrm{ad}_{A}(x)\big)\delta(y)
        -(Q\otimes Q)\big(\mathrm{ad}_{A}(y)\otimes\mathrm{id}+
        \mathrm{id}\otimes\mathrm{ad}_{A}(y)\big)\delta(x),\\
        &(\mathrm{id}\otimes\mathrm{id}-\tau)\big(\mathrm{id}\otimes L_{\prec_{A}}(x)\big)\eta(y)
        =(\mathrm{id}\otimes\mathrm{id}-\tau)\big(\mathrm{id}\otimes L_{\prec_{A}}(x)\big)(Q\otimes\mathrm{id})\delta(y)\\
        &=(\mathrm{id}\otimes\mathrm{id}-\tau)\big(\mathrm{id}\otimes Q\mathrm{ad}_{A}(x)\big)(Q\otimes\mathrm{id})\delta(y)
        =(\mathrm{id}\otimes\mathrm{id}-\tau)(Q\otimes Q)\big(\mathrm{id}\otimes \mathrm{ad}_{A}(x)\big)\delta(y)\\
        &\overset{\eqref{eq:coLie}}{=}(Q\otimes Q)\big(\mathrm{ad}_{A}(x)\otimes\mathrm{id}
        +\mathrm{id}\otimes\mathrm{ad}_{A}(x)\big)\delta(y).
    \end{align*}
    Thus \eqref{eq:bialg1} holds. Similarly, \eqref{eq:bialg2}-\eqref{eq:bialg4} hold. Therefore  $(A,\succ_{A},\prec_{A},\vartheta,\theta)$ is a \sctplb.
\end{proof}

In the following, we shall give an interpretation of Proposition \ref{pro:5.9} in terms of Manin triples.

\begin{defi}\cite{Cha}
Let $(A,[-,-]_{A})$ and $(A^{*},[-,-]_{A^{*}})$ be Lie algebras.
If there is a Lie algebra $(A\oplus A^{*},[-,-]_{d})$ which contains $(A,[-,-]_{A})$ and $(A^{*},[-,-]_{A^{*}})$ as Lie subalgebras, and the natural nondegenerate symmetric bilinear form $\mathcal{B}_{d}$ given by \eqref{eq:bfds} is invariant on $(A\oplus A^{*},[-,-]_{d})$, then we say that $\big((A\oplus A^{*},[-,-]_{d},\mathcal{B}_{d}),(A,[-,-]_{A}),(A^{*}$,
$[-,-]_{A^{*}})\big)$ is a 
{\bf Manin triple of Lie algebras}.
\end{defi}

\begin{thm}\cite{Cha}
    Let $(A,[-,-]_{A})$ be a Lie algebra. Suppose that there is a Lie algebra structure $(A^{*},[-,-]_{A^{*}})$ on the dual space $A^{*}$ and $\delta:A\rightarrow A\otimes A$ is the linear dual of
    $[-,-]_{A^{*}}$.
    Then $(A,[-,-]_{A},\delta)$ is a Lie bialgebra if and only if there is a Manin triple of Lie algebras
    $\big((A\oplus A^{*},[-,-]_{d}$,
    $\mathcal{B}_{d}),(A,[-,-]_{A}),
    (A^{*},[-,-]_{A^{*}})\big)$.
    In this case, the multiplication $[-,-]_{d}$ on $A\oplus A^{*}$ is given by
    \begin{equation}\label{eq:commassomul}
        [x+a^{*},y+b^{*}]_{d}=[x,y]_{A}+\mathrm{ad}^{*}_{A^{*}}(a^{*})y-\mathrm{ad}^{*}_{A^{*}}(b^{*})x+
        [a^{*},b^{*}]_{A^{*}}+\mathrm{ad}^{*}_{A}(x)b^{*}-\mathrm{ad}^{*}_{A}(y)a^{*},
    \end{equation}
    for all $x,y\in A, a^{*},b^{*}\in A^{*}$.
\end{thm}

\begin{defi}
    Let $\big( (A\oplus A^{*},[-,-]_{d},\mathcal{B}_{d}),(A,[-,-]_{A}),(A^{*},[-,-]_{A^{*}})
    \big)$ be a Manin triple of Lie algebras.
    Suppose that $P:A\rightarrow A$ is an averaging operator on
    $(A,[-,-]_{A})$ and $Q^{*}:A^{*}\rightarrow A^{*}$ is an averaging operator on $(A^{*},[-,-]_{A^{*}})$. If $P+Q^{*}$ is an averaging operator on $(A\oplus A^{*},[-,-]_{d})$, then we say that $\big((A\oplus A^{*},[-,-]_{d},P+Q^{*},\mathcal{B}_{d}),(A,[-,-]_{A},P),$
    $(A^{*},[-,-]_{A^{*}},Q^{*})
    \big)$ is a {\bf Manin triple of averaging Lie algebras}.
\end{defi}

\begin{thm}\label{thm:2.11}
Let $\big( (A\oplus
A^{*},[-,-]_{d},\mathcal{B}_{d}),(A,[-,-]_{A}),(A^{*},[-,-]_{A^{*}})
\big)$ be a Manin triple of Lie algebras and the corresponding Lie
bialgebra be $(A,[-,-]_{A},\delta)$. Let $P,Q:A\rightarrow A$ be
linear maps. Then the triple $\big((A\oplus
A^{*},[-,-]_{d},P+Q^{*},\mathcal{B}_{d}),(A,[-,-]_{A},P), (A^{*},
[-,-]_{A^{*}},Q^{*})\big)$ is a Manin triple of averaging Lie
algebras
if and only if \eqref{eq:Ao}, \eqref{eq:ao pair}, \eqref{eq:aoco1}
and \eqref{eq:aoco2} hold, that is,
        $(A$,
        $[-,-]_{A},\delta,P,Q)$ is an averaging Lie bialgebra.
\end{thm}

\begin{proof}
   Let $x,y\in A, a^{*},b^{*}\in A^{*}$. Then we have {\small
\begin{eqnarray*}
     [(P+Q^{*})(x+a^{*}),(P+Q^{*})(y+b^{*})]_{d}
    &\overset{\eqref{eq:commassomul}}{=}&[P(x),P(y)]_{A} +\mathrm{ad}^{*}_{A^{*}}\big(Q^{*}(a^{*})\big)P(y)
    -\mathrm{ad}^{*}_{A^{*}}\big(Q^{*}(b^{*})\big)P(x)\\
    &&+[Q^{*}(a^{*}),Q^{*}(b^{*})]_{A^{*}}
    +\mathrm{ad}^{*}_{A}\big(P(x)\big)Q^{*}(b^{*})
    -\mathrm{ad}^{*}_{A}\big(P(y)\big)Q^{*}(a^{*}),\\  (P+Q^{*})[(P+Q^{*})(x+a^{*}),y+b^{*}]_{d}
    &\overset{\eqref{eq:commassomul}}{=}&P[P(x),y]_{A}
    +P\Big(\mathrm{ad}^{*}_{A^{*}}\big(Q^{*}(a^{*})\big)y\Big)
    -P\big(\mathrm{ad}^{*}_{A^{*}}(b^{*})P(x)\big)\\
    &&+Q^{*}[Q^{*}(a^{*}),b^{*}]_{A^{*}}
    +Q^{*}\Big(\mathrm{ad}^{*}_{A}\big(P(x)\big)b^{*}\Big)
    -Q^{*}\big(\mathrm{ad}^{*}_{A}(y)Q^{*}(a^{*})\big).
\end{eqnarray*}}Hence $P+Q^{*}$ is an averaging operator on $(A\oplus
A^{*},[-,-]_{d})$ if and only if \eqref{eq:Ao}
and the following equations hold:
    \begin{eqnarray}
    	&&[Q^{*}(a^{*}),Q^{*}(b^{*})]_{A^{*}}=Q^{*}[Q^{*}(a^{*}),b^{*}]_{A^{*}}, \label{eq:mp ao1}\\
        &&\mathrm{ad}^{*}_{A^{*}}\big(Q^{*}(a^{*})\big)P(x)
        =P\Big(\mathrm{ad}^{*}_{A^{*}}\big(Q^{*}(a^{*})\big)x\Big)
        =P\big(\mathrm{ad}^{*}_{A^{*}}(a^{*})P(x)\big),\label{eq:pq1}\\
        &&\mathrm{ad}^{*}_{A}\big(P(x)\big)Q^{*}(a^{*})
        =Q^{*}\Big(\mathrm{ad}^{*}_{A}\big(P(x)\big)a^{*}\Big)
        =Q^{*}\big(\mathrm{ad}^{*}_{A}(x)Q^{*}(a^{*})\big).\label{eq:pq2}
    \end{eqnarray}
It is straightforward to show that
\begin{eqnarray*}
\eqref{eq:mp ao1}\Longleftrightarrow \eqref{eq:aoco1},\;
\eqref{eq:pq1}\Longleftrightarrow \eqref{eq:aoco2},\;
\eqref{eq:pq2}\Longleftrightarrow\eqref{eq:ao pair}.
\end{eqnarray*}
Hence the conclusion follows.
\end{proof}

\begin{pro}\label{pro:5.2}
Let $\big( (A\oplus A^{*},[-,-]_{d},P+Q^{*},\mathcal{B}_{d}),
(A,[-,-]_{A},P),(A^{*},[-,-]_{A^{*}},Q^{*}) \big)$ be a Manin
triple of averaging Lie algebras. Then there is a Manin triple of
Leibniz algebras $\big((A\oplus A^{*}$,
$\circ_{d},\mathcal{B}_{d}), (A,\circ_{A}),
(A^{*},\circ_{A^{*}})\big)$ associated to the nondegenerate
symmetric left-invariant bilinear form $\mathcal{B}_{d}$, where
\begin{eqnarray}\label{eq:mpequ1}
(x+a^{*})\circ_{d}(y+b^{*})
:=[(P+Q^{*})(x+a^{*}),y+b^{*}]_{d},\;\forall x,y\in A, a^{*},b^{*}\in A^{*}.
\end{eqnarray}
Moreover, there is a Manin triple $\big(  (
A\oplus
A^{*},\succ_{d},\prec_{d},\mathcal{B}_{d}),
(A,\succ_{A},\prec_{A}),(A^{*},
\succ_{A^{*}}, \prec_{A^{*}})\big)$ of \sctplas, where
\begin{eqnarray}
    &&(x+a^{*})\succ_{d}(y+b^{*})
    :=[(P+Q^{*})(x+a^{*}),y+b^{*}]_{d}-(Q+P^{*})
    ([x+a^{*},y+b^{*}]_{d}),\ \ \ \ \ \label{eq:dc1}\\
    &&(x+a^{*})\prec_{d}(y+b^{*}):=(Q+P^{*})([x+a^{*},y+b^{*}]_{d})
    ,\;\forall x,y\in A, a^{*},b^{*}\in A^{*}.\label{eq:dc2}
\end{eqnarray}
\end{pro}

\begin{proof}
It is clear that the induced Leibniz algebra $(A\oplus
A^{*},\circ_{d})$ defined by \eqref{eq:mpequ1} from the averaging
Lie algebra $(A\oplus A^{*},[-,-]_{d},P+Q^{*})$ contains the
induced Leibniz algebras $(A,\circ_{A}=\circ_{d}|_{A\otimes A})$ and
$(A^{*},\circ_{A^{*}}=\circ_{d}|_{A^*\otimes A^{*}})$ as Leibniz subalgebras.
The left-invariance of $\mathcal{B}_{d}$ follows from Proposition
\ref{pro:1926}. Moreover, by Propositions \ref{pro:equ}, there is
a Manin triple of \sctplas given by \eqref{eq:3603} and
\eqref{eq:3605}. Explicitly, note that the adjoint map of
$P+Q^{*}$ associated to $\mathcal{B}_{d}$ is $Q+P^{*}$. Then we
have
\begin{eqnarray*}
\mathcal{B}_{d}\big( (x+a^{*})\prec_{d}(y+b^{*}), z+c^{*} \big)&\overset{\eqref{eq:3603}}{=}&-\mathcal{B}_{d}\big(
x+a^{*},(z+c^{*})\circ_{d}(y+b^{*}) \big)\\
&=&-\mathcal{B}_{d}\big(
x+a^{*},[(P+Q^{*})(z+c^{*}),y+b^{*}]_{d}\big)\\
&=&\mathcal{B}_{d}\big( [x+a^{*},y+b^{*}]_{d},(P+Q^{*})(z+c^{*}) \big)\\
&=&\mathcal{B}_{d}\big( (Q+P^{*})[x+a^{*},y+b^{*}]_{d}, (z+c^{*}) \big).
\end{eqnarray*}
By the nondegeneracy of $\mathcal{B}_{d}$, we obtain \eqref{eq:dc2}, and   \eqref{eq:dc1} follows from \eqref{eq:3605}.
\end{proof}

Combining Propositions \ref{pro:equ}, \ref{pro:5.9}, \ref{pro:5.2} and Theorems \ref{thm:Manin triple}, \ref{thm:2.11}, we have the following commutative diagram.

\vspace{-.4cm}

{\tiny
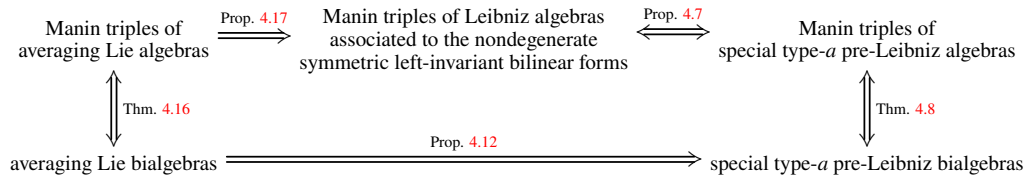
    \begin{equation*}
\begin{split}
      \xymatrix{
         \txt{Manin triples of \\ averaging Lie algebras}
            \ar@{=>}[r]^-{{\rm Prop.}~\ref{pro:5.2}}
            \ar@{<=>}[d]^-{{\rm Thm.}~\ref{thm:2.11}}
              & \txt{Manin triples of Leibniz algebras \\ associated to the nondegenerate \\ symmetric left-invariant bilinear forms}
              \ar@<.4ex>@{<=>}[r]^-{{\rm Prop.}~\ref{pro:equ}}
              & \txt{Manin triples of \\ \sctplas}
              \ar@{<=>}[d]^-{{\rm Thm.}~\ref{thm:Manin triple}}
              \\  
            \txt{averaging Lie bialgebras}
            \ar@{=>}[rr]^-{{\rm Prop.}~\ref{pro:5.9} }
            & & \txt{\sctplbs}
            }
\end{split}
\end{equation*}
}

\noindent{\bf Acknowledgments.} This work is supported by NSFC
(11931009, 12271265, 12261131498, 12326319, 12401031, W2412041),
the Postdoctoral Fellowship Program of
CPSF  (GZC20240755, 2024T005TJ, 2024M761507),
the Fundamental Research Funds for the Central Universities and Nankai Zhide Foundation.

\noindent
{\bf Declaration of interests. } The authors have no conflict of interest to declare.

\noindent
{\bf Data availability. } No data were created or analyzed.

\end{document}